\documentclass[a4paper,11pt]{article}
\usepackage[T1]{fontenc}
\usepackage{ae}
\usepackage{aecompl}
\usepackage{lmodern}
\usepackage[utf8]{inputenc}
\usepackage[english]{babel}
\usepackage[a4paper,scale=0.75]{geometry}
\usepackage[ruled]{algorithm2e} 
\usepackage{amsthm}
\usepackage{amsmath}
\usepackage{amsfonts}
\usepackage{amssymb}
\usepackage{mathrsfs}
\usepackage{url}
\usepackage{makeidx}
\author{Enea Milio}
\title{A quasi-linear time algorithm for computing modular polynomials in dimension $2$}
\date{}
\newcommand{\Z}{\mathbb{Z}}
\newcommand{\ssi}{if and only if }
\newcommand{\Q}{\mathbb{Q}}
\newcommand{\C}{\mathbb{C}}
\newcommand{\F}{\mathbb{F}}
\newcommand{\HH}{\mathcal{H}}
\newcommand{\FF}{\mathcal{F}_2}
\newcommand{\bb}{\backslash}

\newcommand{\Spg}{\textrm{Sp}_{2g}(\Z)}

\newcommand{\pMat}{\begin{ppsmallmatrix}A&B\\C&D\end{ppsmallmatrix}}
\newcommand{\pMatp}{\begin{ppsmallmatrix}A&pB\\C/p&D\end{ppsmallmatrix}}

\newenvironment{psmallmatrix}{\left [\begin{smallmatrix}}{\end{smallmatrix}\right ]}
\newenvironment{ppsmallmatrix}{\left (\begin{smallmatrix}}{\end{smallmatrix}\right )}
\newcommand{\thetacar}[2]{\theta\begin{psmallmatrix}#1\\ #2\end{psmallmatrix}}

\makeatletter
\newcommand*{\newaliascnt}[2]{%
  \begingroup
    \def\AC@glet##1{%
      \global\expandafter\let\csname##1#1\expandafter\endcsname
        \csname##1#2\endcsname
    }%
    \@ifundefined{c@#2}{%
      \@nocounterr{#2}%
    }{%
        \AC@glet{c@}%
        \AC@glet{the}%
        \AC@glet{theH}%
        \AC@glet{p@}%
        \expandafter\gdef\csname AC@cnt@#1\endcsname{#2}%
        \expandafter\gdef\csname cl@#1\expandafter\endcsname
        \expandafter{\csname cl@#2\endcsname}%
    }%
  \endgroup
}
\makeatother
\newtheorem{theo}{Theorem}
\newaliascnt{algocf}{theo}

\newtheorem{deff}[theo]{Definition}
\newtheorem{prop}[theo]{Proposition}
\newtheorem{conj}[theo]{Conjecture}
\newtheorem{lemm}[theo]{Lemma}

\theoremstyle{definition}
\newtheorem{remm}[theo]{Remark}
\newtheorem{exem}[theo]{Example}
\newtheorem{supp}[theo]{Assumption}

\newcommand{\lquo}[2]{\leavevmode\kern-.1em\lower.25ex\hbox{$#2$}\kern-.1em\backslash\kern-.1em\raise.2ex\hbox{$#1$}}
\newcommand{\rquo}[2]{\leavevmode\kern-.1em\raise.2ex\hbox{$#1$}\kern-.1em/\kern-.1em\lower.25ex\hbox{$#2$}}

\begin{document}

\maketitle
%\tableofcontents
\begin{abstract}We propose to generalize the work of Régis Dupont for computing modular polynomials in dimension $2$ to new invariants. 
We describe an algorithm to compute modular polynomials for invariants derived from theta constants and prove under some heuristics that this algorithm is quasi-linear
in its output size. Some properties of the modular polynomials defined from  quotients of theta constants are analyzed.
We report on experiments with our implementation.
\end{abstract}

An isogeny is a morphism between two abelian varieties that is surjective and has a finite kernel. It is an important notion for the theorical study of abelian varieties, but also 
for cryptographic applications because it allows one to transfer the discrete logarithm problem from a variety, where the problem is difficult, to an isogenous variety, where it may be 
easier. %smith

The computation of an isogeny could mean several things: given a maximal isotropic subgroup of the $\ell$-torsion, to be able to compute  the isogenous variety; to calculate the image of a point 
by an isogeny; to check if two abelian varieties are isogenous and if so compute an isogeny. But what interests us here is the computation of all the isogenous varieties (for a fixed degree) of 
a given variety and this can be done with modular polynomials. %Modular polynomials parameterize isogenies.

%Pour certaines applications (calcul de l’anneau d’endomorphisme par exemple), il suffit de construire des graphes d’isogénies et seule la connaissance du polynôme modulaire Φ l est nécessaire. %damien page 28/12

These polynomials also have other applications.
In dimension $1$, they are the key (SEA) for speeding up  the algorithm of Schoof for counting the number of points 
on an elliptic curve (see \cite{Elkies,CountPoint}), % Rappelons que connaˆ le nombre de points sur une vari ́t ́ permet de garantir la s ́curit ́ de la vari ́t ́.
 for constructing elliptic curves with a known number of points by complex multiplication (see \cite{CompHilbClassPol,ClassInvCRT,HilbCRT}) and for the computation of the 
endomorphism ring of elliptic curves (see \cite{EndEllFF}). % (construction of graph of isogenies).
They can be computed in quasi-linear time (see \cite{ModPolVol,Enge}).

In dimension $2$, these polynomials play the same role, but are harder to compute. 
They could also speed up the CRT-algorithm (see \cite{CRTG2}) to compute class fields of degree 4 CM-fields, which would lead to faster algorithms to construct cryptographically secure Jacobians of hyperelliptic curves.
% We have of an efficient arithmetic on the abelian surfaces. %(their have big degrees and coefficients and multivariate).

An algorithm to compute the modular polynomials in dimension $2$ has been introduced by Régis Dupont (see \cite{Dupont}) in $2006$.
Using it, he managed to compute the polynomials parameterizing $(2,2)$-isogenies, but these polynomials are so big that he could not 
compute them for $(3,3)$-isogenies. 
This is due to the fact that he used Igusa invariants (see Definition~\ref{defIg}). 

In this article, we will present a generalization of his algorithm allowing one to compute modular polynomials with invariants $f_1$, $f_2$, $f_3$ derived from theta constants, for a congruence subgroup $\Gamma$
of the symplectic group $\Gamma_2$.
We present results with Streng invariants (see Definition~\ref{defStr}) and quotients of theta constants.
The algorithm proceeds by evaluation/interpolation: without knowing the polynomials we are still able to evaluate
them on any values and if we do so on sufficiently many values, we can interpolate to recover the polynomials. The most important difference with the 
dimension 1 case is that we have to interpolate trivariate rational functions rather than univariate polynomials.
This adds difficulties in the evaluation step; we need to choose the arguments in which to evaluate in a specific way.
For this, we will see that we need to be able to find $\Omega$ modulo $\Gamma$
from $f_1(\Omega)$, $f_2(\Omega)$ and $f_3(\Omega)$.

To do that, we deduce from $f_1(\Omega)$, $f_2(\Omega)$ and $f_3(\Omega)$ the Igusa invariants $j_1(\Omega)$, $j_2(\Omega)$, $j_3(\Omega)$ and then we 
apply Mestre's algorithm to obtain a hyperelliptic curve with these Igusa invariants. Using Thomae's formula, numerical integration and the Borchardt mean,
it is possible to obtain $\Omega$ modulo $\Gamma_2$, under Conjecture~\ref{conj}.
Once we have $\Omega$ modulo $\Gamma_2$, we have to find $\Omega$ modulo $\Gamma$. This can be done thanks to the functional equation of the theta constants of Proposition~\ref{funceq}.
It remains to compute the products of Definition~\ref{defPolMod} which ends the evaluation step.

All the computation are done with multiprecision floating point numbers. Explicit bounds on the size of the coefficients of
the modular polynomials are unknown in dimension $2$ (this is already a hard problem in dimension $1$).
Thus our algorithm is heuristic. Under the heuristics and Conjecture~\ref{conj}, we have shown  that our algorithm is quasi-linear (Theorem~\ref{ThPrincipal}). 
In practice, we may double the precision until we manage to find a sufficient precision to compute the polynomials. 
We underline the fact that the computations have been done at high precision so that we required a fast algorithm to compute the theta constants. This algorithm uses
the Borchardt mean.

We first applied the algorithm of Dupont with Streng invariants instead of Igusa invariants to compute modular polynomials. The invariants of Streng are equivalent
to the Igusa ones in the sense that they describe the same moduli space up to birational equivalence (and indeed, there are formula to switch from ones to the others: see (\ref{IgToIg})). We managed to
compute the modular polynomials parameterizing $(2,2)$- and  $(3,3)$-isogenies. The reason why this was possible is that the use of Streng invariants produce much smaller
polynomials in terms of degrees and sizes of the coefficients, and thus the precision of the computation is smaller (as similarly noticed by Streng in \cite{Streng} for the computation
of class polynomials). For example, the modular polynomials for $p=2$ with Streng invariants fill $2.1$ MB compared to $57$ MB with Igusa invariants.

We have then applied our generalized algorithm to $b'_i(\Omega)=\frac{\theta_i(\Omega/2)}{\theta_0(\Omega/2)}$ for $i=1,2,3$,  which are modular functions for the group $\Gamma(2,4)$,
and computed the modular polynomials with these invariants for $p=3$, $p=5$ and $p=7$. As these polynomials fill respectively $175$ KB, $200$ MB and $29$ GB, 
we have no hope of computing the modular polynomials for larger $p$. The polynomials found are much smaller than those with Igusa or Streng invariants. For comparison, for $p=3$ they fill $175$ KB while those with Streng invariants
fill $890$ MB. We explain this by the fact that these polynomials have symmetries (Theorem~\ref{ThSym}) and are sparse (Theorem~\ref{ThDeg}).

The remainder of this article is organised as follows.
In the first section, we recall the theory of abelian varieties, of theta constants and of modular polynomials. In
the second section, we explain how to interpolate multivariate rational functions and we do a complexity analysis. This section is independent of the others.
The algorithm to compute the modular polynomials is given in the third section and some applications of it are 
described in the fourth section. The fifth section is dedicated to a deeper examination of these polynomials (in particular
the symmetries appearing and an interpretation of the denominators of the coefficients of these polynomials). In Section
$6$ we discuss our implementation of the algorithm and finally in Section $7$ we give some examples of 
hyperelliptic curves over finite field with isogenous Jacobians that we managed to compute using our modular polynomials.

%schéma algo:
%INPUT: x1,x2,x3 in C / M=classes G0(l)G2 
%OUTPUT: P(x1,x2,x3)
%1) déduire de x1,x2,x3 les jinv associées à la courbe
%2) algo Mestre: on en déduit une équation polynomiale
%3) on extrait racine et on calcule les th4 avec thomae
%4) EvalTauFromBj
%5) on en déduit le bon tau grâce à un précalcul
%6) on calcule Prod(x-f(lM.tau))

\section{Theory of modular polynomials}

%An abelian variety of dimension $g$ over $\C$ is analytically isomorphic to a complex torus $\C^g/ \Lambda$, where $\Lambda$ is a lattice (a discrete free $\Z$-module of rank $2g$); 
%but the converse is false for $g\ge 2$. The notion of polarization is needed to study the tori which are abelian varieties. 

%We say that a torus $\C^g/\Lambda$ is \emph{polarizable} if it admits a positive definite non-degenerate Riemann form, where 
%a \emph{Riemann form} on this complex torus  is a Hermitian form $H:\C^g\times \C^g\to\C$ with the additional property that $\Im(H(\Lambda,\Lambda))\subseteq \Z$.
%By \cite[Chapter 4]{Birk}, a polarizable tori is a complex abelian variety.

%We focus now on principally polarized abelian varieties . This means that we can consider that $\Lambda$ is of the form $\Omega\Z^{2g}+\Z^{2g}$ (we call $\Omega$ a \emph{period matrix}). 
%This is justified by the fact that every polarized complex abelian variety is isogenous to a principally polarized one (\cite[Proposition 1.39]{Gruenewald}) 

The Siegel upper half-space $\HH_g$ for dimension $g$ is the set of $g\times g$ symmetric matrices over the complex numbers with positive definite  imaginary part.
It is a moduli space for principally polarized abelian varieties (see \cite[Proposition 8.1.2]{Birk}).
Indeed, a principally polarized abelian variety is a torus $\C^g/(\Omega\Z^{2g}+\Z^{2g})$ for $\Omega\in\HH_g$ (which is  called a \emph{period matrix}).

Let $I_g$ denote the identity matrix of size $g$ and $J=\begin{ppsmallmatrix}0&I_g\\-I_g&0\end{ppsmallmatrix}$. We define  the symplectic group of dimension $2g$ as
$\textrm{Sp}_{2g}(\Z)=\{\gamma\in\textrm{Gl}_{2g}(\Z):\,^t\!\gamma J\gamma=J\}.$
%It can easily be shown that if $\gamma=\pMat$, then $\gamma\in\Spg$ \ssi the three following equalities are verified:
%\[\begin{array}{rcl} ^t\!AC&=&\,^t\!CA\\\,^t\!BD&=&\,^t\!DB\\\,^t\!DA-\,^t\!BC&=&I_{g}\end{array}\]
%and since $\Spg$ is closed under transposition of matrices, these three equalities are equivalent to the three following ones:
%\[\begin{array}{rcl} A\,^t\!B&=&B\,^t\!A\\D\,^t\!C&=&C\,^t\!D\\A\,^t\!D-B\,^t\!C&=&I_{g}.\end{array}\]  
It can easily be shown that 
if $\gamma=\pMat$, then $\gamma\in\Spg$ \ssi the following three  equalities are verified
\begin{equation}\label{eqSpg}
\begin{split}
 ^t\!AC&=\,^t\!CA\\\,^t\!BD&=\,^t\!DB\\\,^t\!DA-\,^t\!BC&=I_{g}
\end{split}
\qquad\textrm{which are equivalent to}\qquad
\begin{split}
 A\,^t\!B&=B\,^t\!A\\D\,^t\!C&=C\,^t\!D\\A\,^t\!D-B\,^t\!C&=I_{g}
\end{split}
\end{equation}
because $\Spg$ is closed under transposition of matrices.
Moreover, $\Spg$  acts (from the left) on $\HH_g$ by 
$\pMat\Omega=(A\Omega+B)(C\Omega+D)^{-1}$, the matrix $C\Omega+D$ being invertible for $\pMat\in\Spg$ and $\Omega\in\HH_g$.

%Like in the genus $1$ case, the general linear group $\textrm{Gl}_{2g}(\Z)$ stabilizes the lattice $\Z^{2g}+\Omega\Z^{2g}$ but it is important, now that we consider the polarizations,
%that the Hermitian form is respected. That is what the symplectic group does.

The quotient space $\Spg\bb\HH_g$ is a moduli space for isomorphism classes of principally polarized abelian varieties of dimension $g$ (see \cite[Theorem 8.2.6]{Birk}).
%This means that if the tori associated to the two matrices $\Omega_1,\Omega_2\in\HH_g$ are isomorphic, %def of iso; iso as ab var or as tori ?
% then there exists a matrix $\gamma=\pMat\in\Spg$ such that $\gamma\Omega_1=\Omega_2$. The isomorphism
%is given by the map $z\mapsto\gamma z:=\,^t(C\Omega_1+D)^{-1}z$ (\cite[Proposition 2.3.7]{Cosset}).

We define $\Gamma_g=\Spg$. Note that $-I_{2g}$ acts trivially on $\HH_g$, so that some authors prefer to consider the projective symplectic group.% $\Gamma_g=\Spg/\langle -I_{2g} \rangle$. %PSp2g

\begin{prop}\label{genG2} The group $\Gamma_g$ is generated by $J$ and the $\frac{g(g+1)}2$ matrices 
\[ M_{i,j}=\begin{pmatrix}I_g&m_{i,j}\\0&I_g,\end{pmatrix}\]
where $m_{i,j}$ is the matrix of size $g$  all entries of which are $0$ except for those at $(i,j)$ and $(j,i)$, which are equal to $1$.
\end{prop}
\begin{proof}This is a direct consequence of  \cite[Proposition 6, pages 41--42]{Klingen}.\end{proof}

Let $\mathcal{F}_g\subseteq \HH_g$ be such that $\Omega=(\Omega_{u,v})_{u,v\in[1,g]}$ is in $\mathcal{F}_g$ \ssi $\Omega$ verifies

\begin{enumerate}
\item $|\Re(\Omega_{u,v})|\le \frac 12$ for each $u,v\in [1,g]$;
\item the matrix $\Im(\Omega)$ is reduced in the sense of Minkowski (see \cite[Chapter I.2]{Klingen} for the definition);  %plus précisément  pour la référence??
\item for each $\pMat\in\Gamma_g$, $|\det(C\Omega+D)|\ge 1$.
\end{enumerate}

\begin{remm} The third point has in principle  to be verified for each matrix of $\Gamma_g$. However for the dimensions $1$ and $2$
this condition has to be verified only for a well-known finite set (of cardinality $1$ in dimension $1$ and $19$ in dimension $2$, see \cite[Proposition 3, p.33]{Klingen} and \cite{Gott}).
\end{remm}

The set $\mathcal{F}_g$ is a \emph{fundamental domain} in the sense that for all $\Omega\in\HH_g$, there exists $\gamma\in\Gamma_g$ such that $\gamma\Omega\in\mathcal{F}_g$
and $\gamma$ is unique if $\gamma\Omega$ is an inner point of $\mathcal{F}_g$.

\begin{deff} Let $\Gamma$ be a subgroup of finite index of $\Gamma_g$ and $k\in\Z$. A \emph{Siegel modular form} of weight
$k$ for $\Gamma$ is a function $f:\HH_g\to\C$ such that
\begin{enumerate}
\item $f$ is holomorphic on $\HH_g$;
\item $\forall\gamma=\pMat\in\Gamma$ and $\Omega\in\HH_g$, $f(\gamma\Omega)=\det(C\Omega+D)^kf(\Omega)$;
\item in the case $g=1$, $f$ has to be holomorphic at the cusps (see for example \cite[Definition 2.3.1]{Schertz} for the definition).
\end{enumerate}
\end{deff}

\begin{deff} Let $\Gamma$ a subgroup of finite index of $\Gamma_g$. A function
$f:\HH_g\to\C$ is a \emph{Siegel modular function} for $\Gamma$ \ssi there are two Siegel modular forms
$f_1$ and $f_2$ for $\Gamma$ of the same weight and such that $f=\frac{f_1}{f_2}$. 
\end{deff}

From another point of view, a complex torus is an abelian variety \ssi it can be embedded into a projective space. This embedding can be done using theta functions.
We will only focus on the classical theta functions because they provide a projective coordinate system for the principally polarized abelian varieties and 
because these functions can easily be handled computationally.

\begin{deff}
Let $\Omega\in\HH_g$ and let $z$ be a vector of $\C^g$. The Riemann theta function is the function
\[\theta:\C^g\times\HH_g\to\C,\quad(z,\Omega)\mapsto\sum_{n\in\Z^g}\exp(i\pi\,^t\!n\Omega n+2i\pi\,^t\!nz)\]
and for $a,b\in \Q^g$, the (classical) theta function with characteristic $(a,b)$ is
\[\begin{array}{ll}
\theta\begin{psmallmatrix}a\\b\end{psmallmatrix}
(z,\Omega)&=\sum_{n\in\Z^g}\exp(i\pi\,^t\!(n+a)\Omega(n+a)+2i\pi\,^t\!(n+a)(z+b))\\
&=\exp(i\pi\,^t\!a\Omega a+2i\pi\,^t\!a(z+b))\theta(z+\Omega a+b,\Omega).
\end{array}
\] 
\end{deff}

These functions converge absolutely and uniformly on every compact subset of $\C^g\times \HH_g$ due to the fact that the imaginary part of $\Omega$
is positive definite. %G page 22

\begin{prop} Let $\Omega\in\HH_g$ and $n\in\mathbb{N}$. The functions $f:\C^g\to \C$ satisfying for each $z\in\C^g$ and each
$m',m''\in\Z^g$, \[f(z+\Omega m'+m'')=f(z)\exp(-i\pi n\,^t\!m'\Omega m'-2i\pi n\,^t\!zm')\]
are said to be functions of level $n$. They form a vector space  of dimension $n^g$ denoted by R$_n^{\Omega}$.
\end{prop}\begin{proof}See \cite[Section II.1]{Mumford83}.\end{proof}

\begin{theo}[Lefschetz] For $n\ge 3$, any linearly independent set of $k\ge n^g$ functions of level $n$  provides an embedding of $\C^g/(\Omega\Z^{2g}+\Z^{2g})$ into $\mathbb{P}^{k-1}(\C)$.
For $n=2$, the functions of level $2$ map only to $\big(\C^g/(\Omega\Z^{2g}+\Z^{2g})\big)/\sim$, where~$\sim$ is the equivalence relation such that $z\sim -z$.
\end{theo}
\begin{proof}See \cite[Page 29]{Mum70}.\end{proof}
Several bases for R$_n^{\Omega}$ are well-known. We use $\mathscr{F}_n=\left\{\thetacar{0}{b}(z,\frac{\Omega}n),b\in \frac 1n\Z^g/\Z^g\right\}$, and $\mathscr{F}_{(n,1)^n}=\left\{\thetacar{0}{b}(z,\Omega)^n,b\in \frac 1n\Z^g/\Z^g\right\}$.
%\paragraph{}
Theta functions have a lot of properties. The following two  are  useful for an implementation of these functions. For a matrix $X$, denote by $X_0$ the vector composed of the diagonal entries of $X$.
\begin{prop} \label{funceq} Let  $\gamma=\pMat\in\Gamma_g$, $e'=\frac 12(\,^t\!AC)_0$ and $e''=\frac 12(\,^t\!DB)_0$.
Then for all vectors $a,b$ in $\mathbb{Q}^g$,  $z$ in $\C^g$ and $\Omega$ in $\HH_g$ we have
\[\begin{array}{cl}
\thetacar{a}{b}(\gamma z,\gamma\Omega)=&\zeta_{\gamma}\sqrt{\det(C\Omega+D)}\exp\left(i\pi\,^t\!z(C\Omega+D)^{-1}Cz\right)
\theta\left[\,^t\!\gamma\begin{ppsmallmatrix}a\\b\end{ppsmallmatrix}+\begin{ppsmallmatrix}e'\\e''\end{ppsmallmatrix}\right](z,\Omega)\\
&\cdot\exp (-i\pi\,^t\!aA\,^t\!Ba)\exp(-i\pi\,^t\!bC\,^t\!Db)\exp(-2i\pi\,^t\!aB\,^t\!Cb)\\
&\cdot\exp(-2i\pi\,^t\!(\,^t\!Aa+\,^t\!Cb+e')e'')\\
\end{array}\] %quelle racine de l'unité ??

where $\zeta_\gamma$ is an eighth root of unity depending only on $\gamma$. 
\end{prop}
\begin{proof}See \cite[Chapter 5, Theorem 2]{Igu72} or \cite[Proposition 3.1.24]{Cosset}.\end{proof}
 \begin{remm}
The eighth root of unity and the square root do not depend on the characteristic.
As we will always consider quotients of theta functions, we do not need to know the root and the determination of the square root.
\end{remm}
%Note that this equation is useless in the case we have $p\gamma\Omega$ instead of $\gamma\Omega$ for an integer $p$. %développer
Note that a matrix acts on the characteristic. The next proposition (\cite[Page 123]{Mumford83}) 
allows us to speak about permutations.

\begin{prop}
For each $\alpha,\beta\in\Z^g$ and $a,b\in\Q^g$, $\thetacar{a+\alpha}{b+\beta}(z,\Omega)=\exp(2i\pi\,^t\!a\beta)\thetacar{a}{b}(z,\Omega).$
\end{prop}

%In particular, the function $\theta^2(0,\Omega)$ is a modular form of weight $1$ for the group $\Gamma(4,8)$.  %def Gamma(4,8)   préferer prop G ?
%Coro 5.11 Mum83

The \emph{theta constants} of level $n$ are the theta functions of level $n$ evaluated at $z=0$. % They give the coordinates of the abelian variety.
In the following, we will focus on the theta constants of genus $2$ with characteristic in $\{0,\frac 12\}^2$.

To simplify the notation we define for all $a=\begin{ppsmallmatrix}a_0\\a_1\end{ppsmallmatrix}$ and $b=\begin{ppsmallmatrix}b_0\\b_1\end{ppsmallmatrix}$ in $\{0,1\}^2$
\[\theta_{b_0+2b_1+4a_0+8a_1}(\Omega):=\theta_{a,b}(\Omega):=\thetacar{a/2}{b/2}(0,\Omega).\]

We have the property that $\theta_{a,b}(\Omega)=(-1)^{^t\!ab}\theta_{a,b}(\Omega)$ so that of the $16$ theta constants, $6$ are identically zero (we say that they are odd) and %Cprop 3.1.2
we denote $\mathcal{P}=\{0,1,2,3,4,6,8,9,12,15\}$ the subscripts of the even theta constants.

The next proposition (\cite[Chapter IV, Theorem $1$]{Igu72}) %(\cite[Proposition 5.5]{Dupont})
 establishes a relation between the $\theta_i^2(\Omega)$ for $i=0,\ldots,15$ and the $\theta_i(\Omega/2)$ for $i=0,\ldots,3$.
\begin{prop}[Duplication formula]\label{dupl} For all $a,b\in\{0,1\}^2$ and $\Omega\in\HH_2$, we have
\[\theta^2_{a,b}(\Omega)=\frac 1{4}\sum_{b_1+b_2\equiv b\bmod 2}(-1)^{^t\!ab_1}\theta_{0,b_1}(\Omega/2)\theta_{0,b_2}(\Omega/2).\]
\end{prop} 

%In his thesis, Régis Dupont has studied the values that the theta constants can take and he has given many bounds. We have retained only two of them (\cite[Proposition 6.1 and Corollary 6.1]{Dupont}).
%\begin{prop} \label{signe}For all $\Omega\in\FF$ and $j\in [0,3]$, $|\theta_j(\Omega)-1|\le 0.405$.
%\end{prop}
%This last proposition is in particular useful in the case where we know $\theta_j^2(\Omega)$, for some $j\in[0,3]$ and $\Omega\in\FF$, because it provides $\theta_j(\Omega)$.

\begin{prop}\label{zerodiag} Let $\Omega\in\HH_2$ and $\Omega'\in\mathcal{F}_2$ be in the same class for the action of $\Gamma_2$. Then either the matrix $\Omega'$ is diagonal
and then exactly one of the even theta constants evaluated in $\Omega$ vanishes and at the same time $\theta_{15}(\Omega')=0$, or $\Omega'$ is not diagonal
and no even theta constant vanishes in $\Omega$ (nor in $\Omega'$).
\end{prop}
\begin{proof}
See \cite[Proposition 6.5 and Corollary 6.1]{Dupont}.
\end{proof}

Let 
\[h_4=\sum_{i\in\mathcal{P}}\theta_i^8, \quad  h_6=\sum_{60\textrm{ triples }(i,j,k)\in\mathcal{P}^3}\pm(\theta_i\theta_j\theta_k)^4,\]\[
  h_{10}=\prod_{i\in\mathcal{P}}\theta_i^2, \quad h_{12}=\sum_{15\textrm{ tuples }(i,j,k,l,m,n)\in\mathcal{P}^6}(\theta_i\theta_j\theta_k\theta_l\theta_m\theta_n)^4,\]
\[  \textrm{and}\quad h_{16}=\frac 13 (h_{12}h_4-2h_6h_{10}). \]
(see for example \cite{Streng,Dupont,Weng} for the exact definition).
The functional equation of Proposition~\ref{funceq} shows that $h_i$ is a Siegel modular form of weight $i$ for the group $\Gamma_2$.  
\begin{deff}\label{defIg} We call Igusa invariants or $j$-invariants the functions $j_1,j_2,j_3$ defined by
\[ j_1:=\frac{h_{12}^5}{h_{10}^6}, \quad j_2:=\frac{h_4h_{12}^3}{h_{10}^4},\quad j_3:=\frac{h_{16}h_{12}^2}{h_{10}^4}.\] 
\end{deff}

\begin{theo} The field $K$ of Siegel modular functions in dimension 2 is $\C(j_1,j_2,j_3)$.
\end{theo}\begin{proof}See \cite{Igu62}.\end{proof}

Generically, by \cite{Igu60}, two principally polarized abelian surfaces are isomorphic if and only if they have the same $j$-invariants.
%\paragraph{}

Let $\Gamma$ be a subgroup of $\Gamma_2$ of index $k$. Denote by $\C_\Gamma$ the field of meromorphic functions of $\HH_2$ invariant under the action of $\Gamma$
(it is the function field of $\Gamma\bb\HH_2$).
In particular, $\C_{\Gamma_2}=K$. By  \cite{Freitag}, $\C_\Gamma$ is a finite algebraic extension of degree $k$ of $\C_{\Gamma_2}$.

Let $f$ be a modular function, $\gamma\in\Gamma_2$ and $p$ a prime number. We define the matrix $\gamma_p:=\pMatp$ and the functions  $f^\gamma$, $f_p$ and $f_p^\gamma$ from $\HH_2\to \C$ by
 $f^\gamma(\Omega)=f(\gamma\Omega)$, $f_p(\Omega)=f(p\Omega)$ and $f_p^\gamma(\Omega)=f(p\gamma\Omega)$ respectively.
Let $\Gamma_{0}(p):=\left\{\pMat\in\Gamma_2:C\equiv 0\bmod p\right\}$.

%For a fixed prime $p$ we define the three functions $j_{\ell,p}:\HH_2\to \mathbb{P}^1(\C)$ by $j_{\ell,p}(\Omega):=j_\ell(p\Omega)$.
The three functions $j_{\ell,p}:=(j_\ell)_p$ are invariant under the group $\Gamma_0(p)$. 
Indeed, if $\gamma\in\Gamma_0(p)$, then $p\gamma\Omega=\gamma_p(p\Omega)$ so that $j_{\ell,p}^\gamma(\Omega)=j_\ell(p\gamma\Omega)=j_\ell(\gamma_p(p\Omega))=j_\ell(p\Omega)=j_{\ell,p}(\Omega)$.
In other words,  $\Omega$ is equivalent to $\gamma\Omega$ for $\gamma\in\Gamma_2$, but that does not mean that $p\Omega$ is equivalent to $p\gamma\Omega$: it is the case only if $\gamma$ is in $\Gamma_0(p)$.

Let $C_p$  be a set of representatives of the quotient $\Gamma_2/\Gamma_0(p)$. The period matrices of the $(p,p)$-isogenous varieties of a variety $\Omega$ are the $p\gamma\Omega$ for $\gamma\in C_p$ (by Theorem~$3.2$ of \cite{BL}). 
%Theorem $3.2$ of \cite{BL} says that $\Gamma_0(p)\bb\HH_2$ is in bijection with the $(p,p)$-isogenies, that is why we are interested in this group.

Proposition~$10.1$ of \cite{Dupont} gives
 $C_p$ for each $p$ and it tells us that $[\Gamma_2:\Gamma_0(p)]=p^3+p^2+p+1$.

\begin{lemm}\label{CorpsFctG0p} For a prime $p$, $\C_{\Gamma_0(p)}$ equals $K(j_{\ell,p})$ for $\ell=1,2,3$.
\end{lemm} \begin{proof} See \cite[Lemma 4.2]{BL}.\end{proof}

Note that the functions $j_{\ell}$ have poles at $\Omega\in\HH_2$ such that $h_{10}(\Omega)=0$. This happens when $\theta_i(\Omega)=0$ for some $i$.
By Proposition~\ref{zerodiag}, if $\Omega'\in\mathcal{F}_2$ is equivalent to $\Omega$, then $\Omega'$ is diagonal. We deduce that $\Omega$ corresponds to a product of elliptic curves. %(with the product polarization). %GL page 6-7
So the functions $j_{\ell,p}$ have poles at $\Omega\in\HH_2$ corresponding to varieties that are $(p,p)$-isogenous to a product of elliptic curves.% with the product polarization.

%We define the $p$-th modular polynomial for $j_1$: 
Let $\Phi_{1,p}(X)=\prod_{\gamma\in C_p}(X-j_{1,p}^\gamma)$. It is the minimal polynomial of $j_{1,p}$ over $K$. As the functions $j_{2,p}$ and
$j_{3,p}$ are contained in $K(j_{1,p})=K[j_{1,p}]$ by Lemma~\ref{CorpsFctG0p}, we define $\Phi_{2,p}(X)$, and $\Phi_{3,p}(X)$ to be the monic polynomials in $K[X]$ of degree less than 
$\deg(\Phi_{1,p}(X))$ satisfying $j_{2,p}=\Phi_{2,p}(j_{1,p})$ and $j_{3,p}=\Phi_{3,p}(j_{1,p})$. 

Furthermore  we have for $\ell=2,3$ that 
$\Phi_{\ell,p}(j_{1,p})=\Psi_{\ell,p}(j_{1,p})/\Phi_{1,p}'(j_{1,p})$ where
\[\Psi_{\ell,p}(X)=\sum_{\gamma\in C_p}j_{\ell,p}^\gamma\prod_{\gamma'\in C_p\bb\{\gamma\}}(X-j_{1,p}^{\gamma'}).\]

\begin{deff}\label{defPolModJ}
Let $p$ be a prime number. We call $\Phi_{1,p}(X)$, $\Psi_{2,p}(X)$ and $\Psi_{3,p}(X)$ the modular polynomials for $j_1$, $j_2$ and $j_3$.
\end{deff}

%\begin{theo} For any prime $p$, the modular polynomials $\Phi_{1,p}(X)$, $\Phi_{2,p}(X)$, $\Phi_{3,p}(X)$ lie in the ring $\Q(j_1,j_2,j_3)[X]$.
%\end{theo} 
%\begin{proof}See \cite[Theorem 5.2]{BL}.\end{proof}
 For any prime $p$, the modular polynomials $\Phi_{1,p}(X)$, $\Phi_{2,p}(X)$, $\Phi_{3,p}(X)$ lie in the ring $\Q(j_1,j_2,j_3)[X]$ (see \cite[Theorem 5.2]{BL}).
This is also the case for $\Psi_{\ell,p}(X)$ for $\ell=2,3$ so that we will focus on $\Phi_{1,p}(X)$, $\Psi_{2,p}(X)$ and $\Psi_{3,p}(X)$.
%For our convenience we will sometimes denote the modular polynomial $\Phi_{1,p}(X)$ by $Phi_{1,p}(X,j_1,j_2,j_3)$ (and similarly for $\Psi_{\ell,p}(X)$).
The evaluation map $\C(j_1,j_2,j_3)\to\C$ sending $j_i$ to $j_i(\Omega)$ maps these polynomials to polynomials in $\C[X]$. 
The meaning of $\Phi_{1,p}(X)$ is that its roots evaluated at $\Omega\in\HH_2$ are the
$j_1$-invariants of the principally polarized abelian surfaces that are $(p,p)$-isogenous to the variety $\Omega$. Moreover, if $x$ is such a root, then $(x,\Phi_{2,p}(x),\Phi_{3,p}(x))$ are the 
$j$-invariants of a principally polarized abelian surface $(p,p)$-isogenous to a variety with invariants $(j_1(\Omega)$, $j_2(\Omega)$, $j_3(\Omega))$.

Denote by $\mathcal{L}_p$ the locus of all the principally polarized abelian surfaces which are $(p,p)$-isogenous to a product of elliptic curves. 
This locus $\mathcal{L}_p$ is a $2$-dimensional algebraic subvariety of the $3$-dimensional
moduli space $\Gamma_2\bb\HH_2$ and can be parameterized by an equation $L_p=0$ for a polynomial $L_p$ in $\Q[j_1,j_2,j_3]$.%pourquoi ?

\begin{lemm}\label{DenIg} The denominators of the coefficients of $\Phi_{1,p}(X)$, $\Psi_{2,p}(X)$ and $\Psi_{3,p}(X)$ are all divisible by the polynomial $L_p$.
\end{lemm}\begin{proof}See \cite[Lemma 6.2]{BL}.\end{proof}
We are particularly interested in the denominators of the modular polynomials because they are at the cause of many difficulties when compute these poynomials.

\section{Interpolation}

We explain in this section how to interpolate multivariate polynomials and rational fractions, which will be needed to compute modular polynomials by evaluation and interpolation.
The problem is the following:
we assume that we have  an algorithm $f$ such that for any $x_1,\ldots,x_n\in\mathbb{C}$
it returns the value $P(x_1,\ldots,x_n)$ (or a floating point approximation thereof), where
$P$ is an unknown multivariate polynomial or rational fraction $P(X_1,\ldots,X_n)$ with complex coefficients. We want to find $P$.

We denote  $\mathcal{M}(d)$ the time to multiply polynomials of degree less than or equal to $d$ with coefficients having $N$ bits and $\mathcal{M'}(N)$ the time complexity to multiply two integers of $N$ bits.
%We have that $\mathcal{M}(d)\in O(d\log{d}\,\mathcal{M'}(N))$ and $\mathcal{M'}(N)\in O(N\log{N}\log{\log{N}})$.
By \cite[Corollary 8.19]{MCA}, we have that $\mathcal{M}(d)\in O(d\log{d}\,\mathcal{M'}(N))$ if we use the FFT and if we assume that $N\in\Omega(\log{d})$, which is necessary to distinguish between the different 
roots of unity used in the FFT. Moreover, $\mathcal{M'}(N)\in O(N\log{N}\log{\log{N}})$ (see \cite{SchnelleMult}).

Following the basic idea of \cite{Dupont}, we work out all the details and give a complexity analysis.
\subsection{Interpolation of a multivariate polynomial}

The problem of interpolating a univariate polynomial $P$ is well-known and can be solved by Lagrange's or Newton's method, which need $\deg(P)+1$ evaluations.
The complexity of fast interpolation is $O(\mathcal{M}(\deg(P))\log(\deg(P)))$ (see \cite[Section 10, Corollary 10.12]{MCA}).
 
%\paragraph{}
In the case of a bivariate polynomial $P(X,Y)$, we notice that it can be written in the following way
\[P(X,Y)=\sum_{i=0}^{d_X}\left(\sum_{j=0}^{d_Y} c_{i,j}Y^j\right)X^i=\sum_{i=0}^{d_X}c_i(Y)X^i.\]
We can compute $P(X,y)$ for a fixed $y$ by evaluating $P(x_i,y)$ for $i=1,\ldots,d_{X}+1$ and interpolating. The
$\ell$-th coefficient of this polynomial is $c_\ell(y)$, which is a univariate polynomial. It can be obtained if one has computed $c_\ell(y_j)$ for $d_Y+1$ values $y_j$.

Thus to obtain $P(X,Y)$ we proceed as follows. For $j$ from $1$ to $d_Y+1$, fix $y_j$ and choose $d_X+1$ values $x_i$ (one may choose the same values for different $j$), then 
evaluate $P(x_i,y_j)$ and interpolate a univariate polynomial to find $P(X,y_j)$. Finally, for each $\ell=0,\ldots,d_X$, interpolate $c_\ell(Y)$.
Hence interpolating a bivariate polynomial needs $(d_Y+1)(d_X+1)$ evaluations and the complexity 
for the interpolation is \[(d_Y+1)O(\mathcal{M}(d_X)\log(d_X))+(d_X+1)O(\mathcal{M}(d_Y)\log(d_Y))\subseteq\tilde{O}(d_Xd_YN).\]

%\paragraph{}
The interpolation of a trivariate polynomial can be done in a similar way. We write it as
\[P(X,Y,Z)=\sum_{i=0}^{d_X}\left(\sum_{j=0}^{d_Y}\left(\sum_{k=0}^{d_Z}c_{i,j,k}Z^k\right)Y^j\right)X^i=\sum_{i=0}^{d_X}\left(\sum_{j=0}^{d_Y} c_{i,j}(Z)Y^j\right)X^i=\sum_{i=0}^{d_X}c_i(Y,Z)X^i.\] 
If, for fixed $y$ and $z$, we evaluate $P(x_i,y,z)$,  $i=1,\ldots,d_X+1$, and then  interpolate, we obtain $P(X,y,z)$ and the
$\ell$-th coefficient is $c_\ell(y,z)$. This gives us a method to evaluate $c_\ell(Y,Z)$ for different values and we can use what we said about
bivariate polynomials to find $c_\ell(Y,Z)$. 

Thus we proceed as follows. For $j$ from $1$ to $d_Y+1$ and for $k$ from $1$ to $d_Z+1$ we evaluate $P(x_i,y_j,z_k)$ for $d_X+1$ values $x_i$ and then 
we do $(d_Y+1)(d_Z+1)$ interpolations to obtain  all the $P(X,y_j,z_k)$.
Each of the $d_X+1$ coefficients $c_\ell(Y,Z)$ is a bivariate polynomial and we have given above the complexity to obtain it.
All in all, we will do $(d_X+1) (d_Y+1) (d_Z+1)$ evaluations and the complexity for the interpolation is
\[(d_Y+1)(d_Z+1)O(\mathcal{M}(d_X)\log(d_X))+(d_X+1)(d_Z+1)O(\mathcal{M}(d_Y)\log(d_Y))+\]\[(d_X+1)(d_Y+1)O(\mathcal{M}(d_Z)\log(d_Z))
\subseteq\tilde{O}(d_Xd_Yd_ZN).\]
%\paragraph{}

We can generalise this improved algorithm  recursively to the case of a polynomial in $n$ variables $X_1,\ldots, X_n$. 
It takes $\prod_{i=1}^n(d_{X_i}+1)$ evaluations. The complexity for the interpolation of a polynomial in $n$ variables is
\[
\sum_{i=1}^n\prod_{\substack{j=1\\j\ne i}}^n(d_{X_j}+1)O(\mathcal{M}(d_{X_i})\log{d_{X_i}})\subseteq\tilde{O}\left(\prod_{i=1}^nd_{X_i}N\right)
\]
%\[\prod_{i=2}^n(d_{X_i}+1)O(\mathcal{M}(d_{X_1})\log(X_1)) + (d_{X_1}+1)\frak{C}(n-1).\]
Note the symmetry which means that the ordering of the variables does not matter.

\subsection{Interpolation of a multivariate rational fraction}

%\textbf{First approach: the univariate case.}

We begin with the univariate case: $F(X)=\frac{A(X)}{B(X)}$, with $A(X)=\sum_{i=0}^{d^A_X}A_iX^i\in\C[X]$ and $B(X)=\sum_{i=0}^{d^B_X}B_iX^i\in\C[X]$. 
We look for the solution with minimal degrees. Each pair $(A,B)$ is then defined only up to a multiplicative constant.
%The most important point to consider here is that the pair $(A,B)$ is defined only up to a multiplicative constant.

Let $n=d^A_X+d^B_X+1$. Writing $A(X)-F(X) B(X)=0$ induces us to proceed with linear algebra: it suffices to evaluate $F$ in $n+1$ 
values $x_i$ and to find the coefficients $A_i$ and $B_i$ by solving the following linear system 
$$
\begin{pmatrix}
1&x_1&\ldots&x_1^{d^A_X}&-F(x_1)&-F(x_1)\cdot x_1&\ldots&-F(x_1)\cdot x_1^{d^B_X}\\
\vdots&\ddots&\vdots&\vdots&\vdots&\vdots&\ddots&\vdots\\
1&x_n&\ldots&x_n^{d^A_X}&-F(x_n)&-F(x_n)\cdot x_n&\ldots&-F(x_n)\cdot x_n^{d^B_X}\\
\end{pmatrix}\cdot
\begin{pmatrix}
A_0\\\vdots\\A_{d^A_X}\\B_0\\\vdots\\B_{d^B_X}
\end{pmatrix}
=\begin{pmatrix}
 0 \\[2ex]
    \vdots \\[2ex]
   0 \\
\end{pmatrix}
$$

%All the difficulty to interpolate a fraction comes from this fact. 

This method is easy to implement, but its complexity is bad.
Another solution consists in using Cauchy interpolation (see \cite[Section 5.8]{MCA}) with the fast Euclidean algorithm (\cite[Section 11]{MCA}),
which produces an algorithm of complexity $O(\mathcal{M}(n)\log(n))$. The number of evaluations is $n$. 

%Let $n$ be an integer. Take $x_0,\ldots,x_{n-1}\in\C$ and $y_i=F(x_i)$ for $0\le i< n$. Let $f$ be an interpolating polynomial.
%We look at polynomials $r(X)$ and $t(X)$ such that $F(X)t(X)=r(X)$ because we have for each $i$, $r(x_i)=t(x_i)y_i=t(x_i)f(x_i)$ \iff $r\equiv tf\bmod (X-x_i)$
%and by the chinese remainder theorem, it is equivalent to ask that $r\equiv tf\bmod m$, where $m=\prod_{i=0}^{n-1}(X-x_i)$. 

We explain it briefly. Let $k$ and $m$ be such that $\deg A<k$ and $\deg B\le m-k$. Take $x_1,\ldots,x_m\in\C$ and $y_i=F(x_i)$ for $1\le i\le m$.
Let $f$ be an interpolating polynomial. We look at polynomials $r(X)$ and $t(X)$ such that for any $i$, $r(x_i)=t(x_i)F(x_i)$; this implies
$(r(x_i)=t(x_i)y_i=t(x_i)f(x_i)$ \ssi $r\equiv tf\bmod (X-x_i)$ for all $i)$ and by the Chinese remainder theorem, it is equivalent to ask that $r\equiv tf\bmod g$, where $g=\prod_{i=1}^{m}(X-x_i)$.

We use then the extended euclidean algorithm on $g$ and $f$. Let $r_j$, $s_j$, $t_j$ be the $j$-th row of the algorithm, where $j$ is minimal such that $\deg r_j<k$ 
(namely $r_1=g,r_2=f,s_1=1,s_2=0,t_1=0,t_2=1$ and $r_\ell=gs_\ell+ft_\ell$ for each row $\ell$).
By Corollary $5.18$ of \cite[Section 5.8]{MCA}, $r_j$ and $t_j$ verify $r_j(x_i)=t_j(x_i)y_i$  and $t_j(x_i)\ne 0$ for all $i$.
%and we can deduce a couple $(r,t)$ such that $r(x_i)=t(x_i)y_i$ and $t(x_i)\ne 0$ for all $i$\ssi $\gcd(r_j,t_j)=1$.
%if in addition $\gcd(r_j,t_j)=1$ then $t_j(x_i)\ne 0$ for all $i$.

Thus, it suffices to compute this row to interpolate the fraction $F$. It is possible to compute a single row with the fast Euclidean algorithm.

%For our purpose, we have used linear algebra because it was fast enough.

%\paragraph{}
%\textbf{Attempts to generalise: the bivariate case.}

We study now the bivariate case $F(X,Y)=\frac{A(X,Y)}{B(X,Y)}$, with 
\[A(X,Y)=\sum_{i=0}^{d^A_X}\sum_{j=0}^{d^A_Y}c_{i,j}^AX^iY^j=\sum_{i=0}^{d_X^A}c_i^A(Y)X^i\in\C[X,Y]\]
 and similarly for $B(X,Y)$. 
One could use linear algebra, but the complexity would be very bad. Thus we would like to proceed as in the bivariate case for polynomials, namely by fixing values $y_j$ and computing
the fractions $F(X,y_j)$ and then by interpolating the coefficients as polynomials in $Y$.

If $F(X,Y)\in\Q[X,Y]$, then for each rational fraction found, one can force the numerator and the denominator to have content $1$, but because of the multiplicative constant, this will not work,
as shown by the next example.
\begin{exem}
Assume that we are searching $F(X,Y)=\frac{3X^2Y^2+Y+2}{3XY+3}$ and that we find  $F(X,1)=\frac{X^2+1}{X+1}$,
$F(X,2)=\frac{12X^2+4}{6X+3}$, $F(X,3)=\frac{27X^2+5}{9X+3}$ and also $F(X,5)=\frac{75X^2+7}{15X+3}$.
For the greater coefficient of the numerator $c^A_2(Y)$: if we interpolate with $y_i=2,3,5$,
which gives us $c_2^A(y_i)=12,27,75$, we find $3Y^2$ which is correct; on the other side, with $y_i=1,2,3$ giving 
$c_2(y_i)=1,12,27$, we obtain the polynomial $2Y^2+5Y-6$, which is completely wrong, because of the simplification
between the numerator and the denominator in the case $Y=1$.
\end{exem}

\begin{remm}\label{simplpol}
For some values of $Y$, there is a simplification by a polynomial; for example $F(X,-5)=-5X-1$. This is something that is immediately 
noticed because the degrees in $X$ are smaller. We assume that every time it happens, we drop the instance.
\end{remm}

In the example, if we had fixed the coefficient of degree $0$
of the denominator of the univariate rational fractions computed at some value ($3$ for instance), the simplification would not have been a problem.
This is not true in general.

\begin{exem} This time write 
$F(X,1)=\frac{X^2+1}{X+1}$, $F(X,2)=\frac{3X^2+1}{\frac{3}{2}X+\frac{3}4}$, $F(X,3)=\frac{\frac{27}5 X^2+1}{\frac 95 X+\frac 35}$,~\ldots, where we always 
fix the coefficient of degree $0$ of the numerator at $1$. Then we deduce by interpolation that $c^A_0(Y)=1$, which is wrong. Indeed, when we divide by a constant 
to obtain the coefficient $1$, it is as if we had divided by $Y+2$.
\end{exem}

%To avoid the error due to a simplification by a constant, we can normalize 
Thus the difficulty is that we have to normalize while being
 sure that the $i$-th coefficient of the numerator and the denominator of each fraction in $X$ comes from the evaluation
of the same polynomial in $Y$.

This normalization is easy to obtain in the very particuliar case where we already know one of the $c_i^A(Y)$ or $c_i^B(Y)\in\C[Y]$
 (different from $0$): we only 
have to multiply the solution found  by the constant which gives us
the good evaluation for the known $c_i$. We can then obtain (using Cauchy interpolation) the fraction with $n(d_Y+1)$ evaluations
and the complexity of the interpolation is
\[(d_Y+1)O(\mathcal{M}(n)\log(n))+(n+1)O(\mathcal{M}(d_Y)\log(d_Y))\subseteq\tilde{O}(d_Xd_YN)\]
where $d_X=\max(d_X^A,d_X^B)$,  $d_Y=\max(d_Y^A,d_Y^B)$ and $n=d_X^A+d_X^B+1\le 2d_X+1$.
\begin{exem}
We continue the preceding example. Assume we know $c_0^A(Y)=Y+2$. We have $c_0^A(1)=3$ and instead of the fraction $\frac{X^2+1}{X+1}$
we take $\frac{3X^2+3}{3X+3}$. We also have $c_0^A(2)=4$ and we write $F(X,2)=\frac{12X^2+4}{6X+3}$ and so on.
%This is the good fraction in the sense that the coefficient appearing come from the evaluation of polynomials
\end{exem}

%In general, an idea to avoid this difficulty is to consider the fraction $F(X,YX)$ (instead of $F(X,Y)$)
%because the coefficient of degree zero in $X$ of its denominator is the constant $c_{0,0}^B$. We can choose to fix it at $1$ if it is not $0$.
%Thus for a given $y$, we evaluate $F(x_i,yx_i)$ for enough $x_i$ and we interpolate. We take the solution where the constant coefficient
%of the denominator is~$1$. Then we can do as for a bivariate polynomial: we interpolate each coefficient $c_\ell^A$ and $c_\ell^B$  in $Y$.
%We thus obtain $F(X,YX)$ and we substitute $Y$ by $Y/X$ to obtain $F(X,Y)$.

In general, an idea to avoid this difficulty is to consider the fraction $F'(X,Y)=F(X,YX)=\frac{A'(X,Y)}{B'(X,Y)}$
 because in this case, $c_0^{B'}(Y)$ is a constant.
If it is not $0$, we can choose to fix it to be $1$ and then the previous argument (one $c_i^{B'}(Y)$ known) holds.
Thus we have $F(X,YX)$ and we substitute $Y$ by $Y/X$ to obtain $F(X,Y)$.
Since $d_X^{A'}=d_T^A$ and $d_X^{B'}=d_T^B$, where the subscript $T$ stands for the total degree, the complexity is $\tilde{O}(d_Td_YN)$ (and $d_T=\max(d_T^A,d_T^B)$).

%Let $n'=d^A_T+d^B_T+1$ where the subscript $T$ stands for the total degree. This method requires  $n(d_Y+1)$ evaluations and the interpolation complexity
%is \[(d_Y+1)O(\mathcal{M}(n')\log(n'))+(n'+1)O(\mathcal{M}(d_Y)\log(d_Y))\]

Note also that in the particular case where the coefficient of degree zero of $c_0^B(Y)$ is $0$, this method does not work.
To overcome this difficulty we can consider $F(X+r,Y+s)$ instead of $F(X,Y)$ for some values $r$ and $s$ such that this coefficient will not be zero.

We study now the trivariate case. We want to interpolate $F(X,Y,Z)=\frac{A(X,Y,Z)}{B(X,Y,Z)}$ with $A(X,Y,Z)$, $B(X,Y,Z)$ in $\C[X,Y,Z]$.
Denote $d_T=\max(d_T^A,d_T^B)$ (and similarly for $d_X$, $d_Y$ and $d_Z$) and $n=d_T^A+d_T^B+1$.
As in the bivariate case, we compute $F(X,XY,XZ)$ and then substitute $Y$ by $Y/X$ and $Z$ by $Z/X$ to obtain $F(X,Y,Z)$.
We explain how to compute $F(X,XY,XZ)$ recursively:

\begin{enumerate}
\item Suppose we are able to compute $F(X,XY,zX)$ for a fixed $z\in\mathbb{C}$. Then we only need  $d_Z+1$ evaluations 
in $z_i$ to interpolate (as polynomials) each coefficient in $Z$ and find $F(X,XY,XZ)$. The number of coefficients is bounded above by $(n+1)(d_Y+1)$ so
that the interpolation complexity for this step is $(n+1)(d_Y+1)O(\mathcal{M}(d_Z)\log(d_Z))$.
\item To obtain $F(X,XY,zX)$ for a fixed $z$, it suffices to apply the interpolation algorithm in the bivariate case. We will do this step $d_Z+1$ times so that
the complexity is\\ $(d_Z+1)((d_Y+1)O(\mathcal{M}(n)\log(n))+(n+1)O(\mathcal{M}(d_Y)\log(d_Y)))$.
\end{enumerate}

In doing this, the number of evaluations will be $n(d_Y+1)(d_Z+1)$ and the final interpolation complexity is $\tilde{O}(d_Td_Yd_ZN)$.
(In the special case where we already know one of the $c_i^A(Y,Z)$ or $c_i^B(Y,Z)$, the complexity will be $\tilde{O}(d_Xd_Yd_ZN)$).
%\paragraph{}

An improvement of this algorithm is obtained in noting that there is the possibility to substitute $Y$ by $Y/X$ in the second step to find $F(X,Y,zX)$ and to compute $F(X,Y,XZ)$
in the first one. Thus, the number of coefficients in the first step will be bounded above by $(n'+1)(d_Y+1)$ where $n'$ is $\deg_X^A(F(X,Y,zX))+\deg_X^B(F(X,Y,zX))+1$,
which is $\le n$, which allows one to reduce the number of interpolations.
The complexity is then 
\[(d_Y+1)(d_Z+1)O(\mathcal{M}(n)\log(n))+(n+1)(d_Z+1)O(\mathcal{M}(d_Y)\log(d_Y))+\]\[
(n'+1)(d_Y+1)O(\mathcal{M}(d_Z)\log(d_Z))\subseteq \tilde{O}(d_Td_Yd_ZN).\]

%\paragraph{}
We can generalize this recursively to the case of a rational fraction $F$ with $m$ variables $X_1,\ldots,X_m$.
We find  
\[\prod_{i=2}^m(d_{X_i}+1)O(\mathcal{M}(n)\log(n))+\sum_{j=2}^m\prod_{\substack{i=2\\i \neq j}}^m(d_{X_i}+1)n(j)O(\mathcal{M}(d_{X_j})\log(d_{X_j}))\subseteq\tilde{O}(d_T\prod_{i=2}^md_{X_i}N)\]
where $n=d_T^A+d_T^B+1$ and $n(j)$ is one plus the degree in $X_1$ of the numerator plus the degree in $X_1$ of the denominator 
of $F(X_1,X_2,\ldots,X_{j-1},X_jX_1,\ldots,X_{m}X_1)$.

Note that all these formulae for the complexity in the case of rational fractions are asymmetric so that the choice of the order of the variables is important.
The formulae suggest that it is preferable to take $X_1$ as the variable with the largest degree. In that case, $n\le 6d_{X_1}+1$ and the complexity of the interpolation 
is then $\tilde{O}(\prod_{i=1}^md_{X_i}N)$.

\section{Evaluation}

We have seen that the modular polynomials lie in the ring $\Q(j_1,j_2,j_3)[X]$ so that we have to interpolate trivariate rational fractions to compute them.
Using the method of interpolation of a rational fraction $F$ exposed in the preceding section requires one to evaluate it at the points  $F(x_i,x_iy_j,x_iz_k)$,
where there exists $\Omega\in\HH_2$ such that $(j_1(\Omega),j_2(\Omega),j_3(\Omega))=(x_i,x_iy_j,x_iz_k)$. 
We present here the method presented in \cite{Dupont} to deduce a matrix $\Omega\in\HH_2$ from its $j$-invariants
and then we present a way to extend this algorithm for other invariants.

\subsection{Computing modular polynomials with the $j$-invariants}\label{algoRegis}
%In practice, the modular polynomials have large coefficients and degrees so that we can not work with integers because the precision needed will be too large. We use floating-point multiprecision
%and the letter $N$ will designate this precision in bits.
In practice, the modular polynomials have large coefficients and degrees. We use floating point multiprecision to compute them
and the letter $N$ will designate this precision in bits.
We have an input $(x,y,z)\in\C^3$ and we are looking for $\Omega\in\HH_2$ such that $(j_1(\Omega),j_2(\Omega),j_3(\Omega))=(x,y,z)$.

The key to do this is to look at the Borchardt mean. 
Let $(z_k)_{k\in \{1,2,3\}}\in\C^3$. We define the Borchardt sequence for $k\in \{1,2,3\}$ by
\[u_0^{(0)}=1\qquad\textrm{and}\qquad u_k^{(0)}=z_k \] 
and recursively for all $n\ge 0$
\[u_0^{(n+1)}=\frac 14\sum_{k=0}^3 u_k^{(n)} \qquad\textrm{and}\qquad
u_k^{(n+1)}=\frac 14\sum_{k_1+k_2\equiv k\bmod 4} v_{k_1}^{(n)}v_{k_2}^{(n)} \]
where $v_0^{(n)}$ is any square root of $u_0^{(n)}$ and  $v_k^{(n)}=0$ if $v_0^{(n)}=0$ or $u_k^{(n)}=0$, otherwise
$v_k^{(n)}$ is the square root of $u_k^{(n)}$ such that $\left|v_0^{(n)}-v_k^{(n)}\right|\le \left|v_0^{(n)}+v_k^{(n)}\right|$ and with $\Im\left(v_k^{(n)}/v_0^{(n)}\right)>0$ if there is equality.

This sequence converges to a unique complex number called the Borchardt mean and denoted by $B_2((z_k)_{k\in\{1,2,3\}})$. 
Let $b_i(\Omega):=\theta_i^2(\Omega)/\theta_0^2(\Omega)$ for $1\le i\le 15$. We have 

\begin{prop}[\cite{Dupont}, Proposition $9.1$]\label{BorThet0} For all $\tau\in\mathcal{F}_2$, $ B_2(b_1(\tau),b_2(\tau),b_3(\tau))=\frac 1{\theta_0^2(\tau)}$.\end{prop} 
Note that from this proposition and the ten even $b_i(\Omega)$, we can deduce all the $\theta_i^2(\Omega)=b_i(\Omega)\theta_0^2(\Omega)$ at the working precision (with some loss of precision).

\begin{conj}[\cite{Dupont}, Conjecture $9.1$]\label{conj} With the notation of Proposition~\ref{genG2} we have, for all $\tau\in\FF$ and for all $\gamma\in\{(JM_{1,1})^2,$  $(JM_{1,2})^2,$  $(JM_{2,2})^2\}$:
$ B_2(b_1(\gamma\tau),b_2(\gamma\tau),b_3(\gamma\tau))=\frac 1{\theta_0^2(\gamma\tau)}.$
%[B_2(b_1(\gamma\tau),b_2(\gamma\tau),b_3(\gamma\tau))=\frac 1{\theta_0^2(\gamma\tau)}.\]
\end{conj}

If this conjecture  is true, it can easily be shown that for $\tau=\begin{ppsmallmatrix}\tau_1&\tau_3\\ \tau_3&\tau_2\end{ppsmallmatrix}$
\begin{equation} \label{eqconj1}
\tau_1=\frac{\imath}{\theta_4^2(\tau)B_2\left(\frac{\theta_0^2(\tau)}{\theta_4^2(\tau)},\frac{\theta_6^2(\tau)}{\theta_4^2(\tau)},\frac{\theta_2^2(\tau)}{\theta_4^2(\tau)}\right)},
\end{equation}
\begin{equation}\label{eqconj2}
\tau_2=\frac{\imath}{\theta_8^2(\tau)B_2\left(\frac{\theta_9^2(\tau)}{\theta_8^2(\tau)},\frac{\theta_0^2(\tau)}{\theta_8^2(\tau)},\frac{\theta_1^2(\tau)}{\theta_8^2(\tau)}\right)},
\end{equation}
and 
\begin{equation}\label{eqconj3}
\tau_3^2-\tau_1\tau_2=\frac{1}{\theta_0^2(\tau)B_2\left(\frac{\theta_8^2(\tau)}{\theta_0^2(\tau)},\frac{\theta_4^2(\tau)}{\theta_0^2(\tau)},\frac{\theta_{12}^2(\tau)}{\theta_0^2(\tau)}\right)}.
\end{equation}

If $\tau$ is in the fundamental domain, the Minkowski reduction implies that $\Im(\tau_3)\ge 0$ which allows us to extract the good square root and obtain $\tau$. %page 135
Thus, it remains to show how to deduce from $(x,y,z)\in\C^3$, the ten $b_i(\Omega)$ where $\Omega\in\mathcal{F}_2$ is such that $(j_1(\Omega),j_2(\Omega),j_3(\Omega))=(x,y,z)$. This can be done in four steps.

\begin{enumerate}
\item 
  The first one is to use  Mestre's algorithm (see for example \cite{Mestre}) at precision $N$ to find a genus $2$ curve  $Y^2=f(X)$ over $\C$ with $f$ having degree $6$ whose Igusa invariants are $(x,y,z)$.
%(Recall that each principally polarized simple abelian surface is the Jacobian of a hyperelliptic curve (\cite[Proposition 2.2.16]{Cosset})).
%and that the Abel-Jacobi map (\cite{Mumford83,Cosset}) is an isomorphism between the Jacobian and the torus corresponding to this variety.

%toujours possible ?  \Im(j_1,j_2,j_3)=C^3 ??
% complexité ?
%autre solution: invariants de Rosenhain

\item  %We now want to find $b_i(\Omega):=\frac{\theta_i^2(\Omega)}{\theta_0^2(\Omega)}$, $i\in\mathcal{P}$ with $\Omega\in\FF$.

Once we have $f$, it is easy to deduce the set $E$ of roots of $f$ at precision $N$ 
 and from this set we use  Thomae's formula (see \cite{Thomae}).
Recall that this formula allows one to obtain the fourth power of the theta constants from an ordering of the roots of $f$ (which corresponds to
 a choice of the basis of the homology group of the Riemann surface of the hyperelliptic curve).

The problem here is that the functions $b_i$ are not invariant under the symplectic group $\Gamma_2$ (but for a subgroup as we will see later). This means 
that for two matrices equivalent under the action of $\Gamma_2$ (namely they have the same $j$-invariants), the evaluation of the $b_i$ in 
these matrices produces differents results.
Hence the theta constants found with Thomae's formula gives us $b_i^2(\gamma\Omega)$ for some unknown $\gamma\in\Gamma_2$.

\item Now use a numerical integration technique (see for example \cite{QuadGauss,Holo,Molin}) at low precision $N'$ with the same choice of the basis of the homology group  
to find the period matrix $\gamma\Omega$  that we reduce into the fundamental domain to obtain $\Omega$ at precision $N'$ and $\gamma$.
Compute $b_i(\Omega)$ at  precision $N'$ (with some algorithm to compute theta constants). 
We do not use a numerical integration technique at precision $N$ because it is too slow and it would increase the complexity of the algorithm.
\item Using the functional equation of Proposition~\ref{funceq} on $b_i^2(\gamma\Omega)$ with the matrix $\gamma^{-1}$ allows one to obtain the $b_i^2(\Omega)$ at precision $N$
and knowing $b_i(\Omega)$ at precision $N'$ is enough to deduce the good square root and obtain $b_i(\Omega)$ at precision $N$.
\end{enumerate}

\begin{supp}\label{supp}
Note that we make the assumption that the numerical integration technique provides some $\gamma\Omega$ with $\gamma$ small enough
such that it can be correctly reduced in the fundamental domain at the precision $N'$.
\end{supp}
%\item Finally we want to find $\Omega$ at precision $N$ using $b_i(\Omega)$ (also at precision $N$).

We thus obtain the following algorithm.
\paragraph{}
\begin{algorithm}[H]
\KwData{$(x,y,z)=(j_1(\Omega),j_2(\Omega),j_3(\Omega))$ for some unknown $\Omega\in\mathcal{F}_2$, the working precision $N$ and a smaller precision $N'$}
\KwResult{$\Omega$}
\BlankLine
\nl Use Mestre's algorithm to obtain a hyperelliptic curve $Y^2=f(X)$ at precision $N$\;
\nl Deduce the ten $b_i(\Omega)$ at precision $N$ using some numerical integration technique at precision $N'$\;
\nl Use Proposition \ref{BorThet0} to obtain the square of the theta constants at the working precision\;
\nl Use (\ref{eqconj1}), (\ref{eqconj2}), (\ref{eqconj3}) to compute $\Omega$ at precision $N$ (with some loss).
\caption{$\Omega$ from $(j_1(\Omega),j_2(\Omega),j_3(\Omega))$}
\label{IgToTau}
\end{algorithm}
\paragraph{}
The second step is Algorithm $12$ of \cite{Dupont} and the third and fourth are Algorithm $13$. 
They have complexity $O(\mathcal{M'}(N))$ and $O(\mathcal{M}'(N)\log(N))$ (where $\mathcal{M'}(N)$ is the time complexity to multiply two integers of $N$ bits) so that 
the algorithm is in $\tilde{O}(N)$. 
%The conjecture has been tested and verified numerically by Dupont for many millions of random matrices. We remark that it is easy to test if the matrix $\Omega$ found at the end  has the good $j$-invariants or not.

\begin{remm} 
Starting from $(1,b_1(\gamma\tau),b_2(\gamma\tau),b_3(\gamma\tau))$ and an approximation or $\tau$, it is possible to compute a Borchardt sequence where the squares roots are choosen
to be the $\frac{\theta_i(2^n\gamma\tau)}{\theta_0(\gamma\tau)}$ at each step. Using the duplication formula (Proposition~\ref{dupl}), it can easily be proved that this will converge to $\frac 1{\theta_0^2(\gamma\tau)}$.
Thus it should be possible to not rely on the conjecture while maintenaing the same complexity, as already stated in the first variant page $200$ of \cite{Dupont} and in the remark after Theorem~$12$ of \cite{cmh}.

The conjecture has been tested and verified numerically by Dupont for many millions of random matrices. We underline that it is easy to test if the matrix $\Omega$ found at the end 
has the good $j$-invariants or not.

\end{remm} 

%complexité algo de mestre ?? 
\subsection{New invariants for the modular polynomials} 

%Thus we have the following general situation.
We begin in giving a generalization of the modular polynomials in order to have the possibility to use other invariants.
In genus~$1$, this goes back to the works of Schläfli and Weber (see \cite{Schlafli} and \cite[Section 4.2 and 4.3]{Enge}).

We consider only the congruence subgroups $\Gamma\subseteq\Gamma_2$, namely the groups with $\Gamma(n)=\{M\in\Gamma_2:M\equiv \pm Id_4\bmod n\}\subseteq \Gamma$ for some $n$. If $n$ is minimal with this property, we say that $n$ is the level of~$\Gamma$.
Let $\Gamma$ be a congruence subgroup and  $f_1$, $f_2$, $f_3$ be three modular functions which are generators for the function field of $\Gamma\bb\HH_2$.
Let $p$ be a prime number such that the level of $\Gamma$ is prime to $p$. 
Let $C_p$ be a set or representatives of $\Gamma/(\Gamma\cap\Gamma_0(p))$.

\begin{deff} \label{defPolMod} The modular polynomials for these data are, for $\ell=2,3$,
\[
\Phi_{1,p}(X)=\prod_{\gamma\in C_p}(X-f_{1,p}^\gamma)\qquad\textrm{and}\qquad
\Psi_{\ell,p}(X)=\sum_{\gamma\in C_p}f_{\ell,p}^\gamma\prod_{\gamma'\in C_p\bb\{\gamma\}}(X-f_{1,p}^{\gamma'}).
\]
\end{deff}

% justifier que dans Q(...) -> serie de fourier

%Thus we are looking at the modular polynomials:
%\[\Phi_{1,p}(X)=\prod_{\gamma\in C_p}\left (X-\frac{\theta_1}{\theta_0}(\frac{p\gamma\Omega}2)\right)\]
%and for $\ell=2,3$:
%\[\Psi_{\ell,p}(X)=\sum_{\gamma\in C_p}\frac{\theta_{\ell}}{\theta_0}(\frac{p\gamma\Omega}2)\prod_{\gamma'\in C_p\bb\{\gamma\}}(X-\frac{\theta_1}{\theta_0}(\frac{p\gamma\Omega}2))\]
%where $C_p$ is a set of representatives of $\Gamma(2,4)/(\Gamma(2,4)\cap\Gamma_0(p))$.

We will sometimes write $\Phi_{1,p}(X,f_1,f_2,f_3)$ instead of $\Phi_{1,p}(X)$ and similarly for $\Psi_{\ell,p}(X)$.
While the interpolation phase is still the same, the evaluation is slightly different: this time we have to find $\Omega\in\HH_2$ from a triple $(x_1,x_2,x_3)\in\C$  such that $f_i(\Omega)=x_{i}$.
Of course, this step depends on the three functions, but we will still give a general algorithm.
On the other side, the computation of $\Phi_{1,p}(X,f_1(\Omega),f_2(\Omega),f_3(\Omega))$ and of $\Psi_{\ell,p}(X,f_1(\Omega),f_2(\Omega),f_3(\Omega))$ for some $\Omega$ does not change. We can apply the same
algorithm and they have the same complexity (except of course for the evaluation of the $f_i(p\gamma\Omega)$).

As in the dimension $1$ case, we have tried to look at modular functions that would produce smaller modular polynomials than those with the $j$-invariants.
%\paragraph{}

The first we tried are the invariants used by Streng in his thesis \cite{Streng} to obtain smaller class polynomials. These invariants for $\Gamma_2$ are defined to have 
the minimal power of $h_{10}$ in the denominators.
%It turns out that there is the same difficulty with the Igusa class polynomials. Yet Streng in his thesis managed to find ``smaller'' invariants.

\begin{deff}\label{defStr}
We call Streng invariants the functions $i_1,i_2,i_3$ defined by
\[i_1:=\frac{h_4h_6}{h_{10}},\quad i_2:=\frac{h_4^2h_{12}}{h_{10}^2},\quad i_3:=\frac{h_4^5}{h_{10}^2}.\]
\end{deff}
We will also say that these are $j$-invariants. The context will make it clear if we are speaking of  Streng invariants  or of Igusa invariants. This is justified by the next theorem.
Note that it is easy to deduce from the Igusa invariants the Streng ones and vice versa.
Indeed, we have
\begin{equation}\label{IgToIg} 
i_1=\frac{j_2(j_2-3j_3)}{2j_1},\quad i_2=\frac{j_2^2}{j_1},\quad i_3=\frac{j_2^5}{j_1^3}\quad\textrm{and}\quad j_1=\frac{i_2^5}{i_3^2},\quad j_2=\frac{i_2^3}{i_3},\quad j_3=\frac{i_2^2(i_2-2i_1)}{3i_3}.
\end{equation}

\begin{theo} The field $K$ of Siegel modular functions in dimension 2 is $K=\C(j_1,j_2,j_3)=\C(i_1,i_2,i_3)$.
\end{theo}
Moreover, we also have the properties for $\ell=1,2,3$ and a prime $p$ that the three $i_{\ell,p}$ are invariants under the group $\Gamma_0(p)$ and that $\C_{\Gamma_0(p)}=K(i_{\ell,p})$ 
(the proof is similar to the one for the invariants of Igusa).

Thus to compute the modular polynomials, the difference with the invariants of Igusa is small: from $(i_1(\Omega),i_2(\Omega),i_3(\Omega))$ to 
obtain $\Omega$, it is sufficient to use (\ref{IgToIg}) to deduce the triple $(j_1(\Omega),j_2(\Omega),j_3(\Omega))$ and then to use Algorithm \ref{IgToTau}. The computation
of the $i_\ell(p\gamma\Omega)$ (for a prime $p$, $\ell=1,2,3$ and $\gamma\in C_p$) is equivalent by (\ref{IgToIg}) to the computation of the $j_\ell(p\gamma\Omega)$.

The modular polynomials with Streng invariants are much smaller in terms of degrees and precision of the coefficients than those with Igusa invariants so that the interpolation step can be done more rapidly
and the number of times we use Algorithm \ref{IgToTau} is also much smaller (see the next section).

%\paragraph{}
Other invariants can be obtained by using the theta constants. This is motivated by the fact that the $j$-invariants are defined in terms of the theta constants.
%We have calculated the polynomials for the group $\Gamma(2)=\{\gamma\in\Gamma_2:\gamma\equiv I_4\bmod 2\}$ (where $I_4$ is the identity matrix of size $4$) and the group $\Gamma_b$ of $\gamma\in\Gamma_2$ which fixes $b_1,b_2$ and $b_3$
% (see \cite{Dupont}).
%In this article, we will explain our results for the group $\Gamma(2,4)$ because it provides the smallest polynomials.

Let $\Gamma(2,4)=\left\{\pMat\in\Gamma_2:\pMat\equiv I_4 \bmod 2\textrm{ and }B_0\equiv C_0\equiv 0\bmod 4\right\}$, which is a normal subgroup of $\Gamma_2$.
It is well-known that the $b_i(\Omega):=\theta_i^2(\Omega)/\theta_0^2(\Omega)$ are modular functions for the group $\Gamma(2,4)$. 
Actually, Theorem~$1$ of \cite{Manni} states that the field $\C_{\Gamma(2,4)}$ of modular functions  belonging to $\Gamma(2,4)$ is $\C(b_1,...,b_{15})$. 

Define for $i=1,2,3$ the functions $b_i'(\Omega):=\frac{\theta_i(\Omega/2)}{\theta_0(\Omega/2)}$.
From these three functions, it is easy to deduce the ten $b_i$ using the duplication formula (Proposition~\ref{dupl}). The converse is also true because we have 
\begin{equation}\label{bequib}
\begin{split}
&b'_1=(b_1+b_9)(1+b_4+b_8+b_{12})^{-1},\\
&b'_2=(b_2+b_6)(1+b_4+b_8+b_{12})^{-1},\\
&b'_3=(b_3+b_{15})(1+b_4+b_8+b_{12})^{-1}.
\end{split}
\end{equation}
Thus we consider the $b'_i$ which allows us to handle three generators instead of ten.

\begin{prop}
Let $p>2$ be a prime number. The classes of $\Gamma(2,4)/(\Gamma_0(p)\cap\Gamma(2,4))$ are in bijection with the classes of $\Gamma_2/\Gamma_0(p)$.
\end{prop} 
\begin{proof}
Consider the map $\phi:\Gamma(2,4)\to \Gamma_2/\Gamma_0(p)$   %with $\phi(x(\Gamma_0(p)\cap\Gamma(2,4)))=x\Gamma_0(p)$. 
with kernel $\Gamma_0(p)\cap\Gamma(2,4)$. The surjectivity comes from the Chinese remainder theorem and the fact that $\textrm{Sp}(4,\Z)\to\textrm{Sp}(4,\Z/4p\Z)$ is surjective
 (the proof of which is analogous to \cite[Section 6.1]{LangEllfct}).
\end{proof}

\begin{prop} For a prime $p>2$, $\C_{\Gamma(2,4)\cap\Gamma_0(p)}$ equals $\C_{\Gamma(2,4)}(b_{i,p}')$ for every $i=1,2,3$.\end{prop} 
\begin{proof}
The proof is similar to that of Theorem~$4.2$ of \cite{BL}. One has to use the isomorphism between $\Gamma(2,4)/(\Gamma(2,4)\cap\Gamma(p))$ and $\Gamma_2/\Gamma(p)$
which comes from the Chinese remainder theorem and the surjectivity of $\textrm{Sp}(4,\Z)\to\textrm{Sp}(4,\Z/4p\Z)$.
\end{proof}

\begin{prop} \label{LieQ}
The modular polynomials for $b'_1$, $b'_2$ and $b'_3$ lie in the ring $\Q(b'_1,b'_2,b'_3)[X]$. More generally, it is also the case for any functions derived from the theta constants.
\end{prop}
\begin{proof}
This comes from the fact that, for any $\Omega=\begin{ppsmallmatrix}\Omega_1&\Omega_2\\ \Omega_2&\Omega_3\end{ppsmallmatrix}$, the functions $b'_1,b'_2$ and $b'_3$ have a Laurent series expansion 
in $q_i=\exp{(2\imath\pi\Omega_i)}$ with rational coefficients. We conclude with a similar proof of Theorem~$5.2$ of \cite{BL}.
\end{proof}

We want to deduce $\Omega$ from $(b'_1(\Omega),b'_2(\Omega),b'_3(\Omega))=(x_1,x_2,x_3)$.
The first thing to do is to deduce from $(x_1,x_2,x_3)$ the Igusa invariants of $\Omega$. 
This can be done easily by calculating  the ten 
$b_i(\Omega)$ from the three $x_i$ with the duplication formula and then by using the definition of the $j$-invariants (see Definition \ref{defIg}). 
We then execute Algorithm \ref{IgToTau} to deduce  a period matrix $\Omega'\in\HH_2$ from the $j$-invariants. 
Unfortunately, this $\Omega'$ is equivalent to $\Omega$ in the sense that they have the same $j$-invariants, but this does not imply that $b'_i(\Omega')=x_{i}$ because
the functions $b_i$ are invariants for the group $\Gamma(2,4)$ (and not $\Gamma_2$).

%\paragraph{}
To overcome this difficulty, we have to consider the cosets of $\Gamma_2/\Gamma(2,4)$.
To find the good $\Omega$ modulo $\Gamma(2,4)$, we can take all the representatives $\gamma$ of this quotient and evaluate the three $b'_i(\gamma\Omega')$ at low precision.
The triple nearest to $(x_1,x_2,x_3)$ gives a matrix $\gamma'$ and then we use the functional equation of Proposition~\ref{funceq} to obtain $b'_i(\gamma'\Omega')=b'_i(\Omega)$ at the working precision $N$.
But this index is large so that this method is slow. Indeed, $[\Gamma_2:\Gamma(2,4)]=11520$.
The next method we propose is faster.
% and it takes too much time to compute the $b'_i(\gamma\Omega')$ for all $\gamma$.

%\paragraph{}
Another solution consists of precomputing the action (permutations and constants) of the representatives of $\Gamma_2/\Gamma(2,4)$ using the 
functional equation and then comparing the three $b'_i(\Omega)$ with the three $b'_i(\Omega')$ to deduce the action and finally to find a representative $\gamma$ in the precalculation which gives us $\gamma\Omega'=\Omega$.
The time spent for each $\gamma$ is thereby negligible.

More precisely, we  know the three $b'_i(\Omega)$ and thus the ten even $b_i(\Omega)$ (recall that it is equivalent by (\ref{bequib}) to have the three $b'_i$) and $\Omega'$ (at the working precision). We compute the ten $b_i(\Omega')$ and
we are looking for $\gamma$ such that $\gamma\Omega'=\Omega$.
Use the functional equation to obtain $\theta_{k}^2(\gamma\Omega')=\zeta_{\gamma}^2\det(...)\imath^{\epsilon(\gamma,k)}\theta_{\ell}^2(\Omega')$ with $\epsilon(\gamma,k)\in\{0,1,2,3\}$. As the $b_i$ are quotients of theta constants,
we can already forget about $\zeta_\gamma^2$ and $\det(...)$. We will say in this case that $k$ is sent to $\ell$ by the action of $\gamma$.
If $0$ is sent to $0$, then the sets $A$ of the ten $b_i(\Omega)=b_i(\gamma\Omega')$ and $B$ of the ten $b_i(\Omega')$ are equal up to permutation and fourth roots of unity. It is easy to compare these two sets to deduce the action
of the matrix $\gamma$.
But the difficulty is that $0$ is not always sent to $0$ and thus $b_i(\gamma\Omega')$ can not be written as a root of unity times $b_j(\Omega')$ (if $i$ is sent to $j$), but as a root times a quotient 
of squares of theta constants evaluated at $\Omega'$.
However note that there is a $c$ such that $c$ is sent to $0$ and a $d$ such that $0$ is sent to $d$. Thus we have $b_c(\gamma\Omega')=\imath^{\epsilon(\gamma,c)-\epsilon(\gamma,0)}\,b_d(\Omega')^{-1}$
and comparing $A$ with $B$ up to a fourth root of unity
it is possible to find $b_d(\Omega')$. Then we have $b_k(\gamma\Omega')=\imath^{\epsilon(\gamma,k)-\epsilon(\gamma,0)}\frac{\theta_\ell^2(\Omega')}{\theta_d^2(\Omega')}=\imath^{\epsilon(\gamma,k)-\epsilon(\gamma,0)} b_\ell(\Omega')b_d(\Omega')^{-1}$
and it is enough to multiply the set $A$ by $b_d(\Omega')^{-1}$ and compare it to $B$ to deduce the action of $\gamma$.
%\paragraph{}

This method can also be used to modify Step $3$ of Algorithm \ref{IgToTau}. Indeed,
in the case that we cannot choose the basis of the homology group for the numerical integration, we obtain the period matrix $\Omega$ at low precision,
but we do not know the matrix $\gamma$ such that $\gamma\Omega$ is the period matrix coming from Thomae's formula.
As explained above by comparing  $b_i(\Omega)$ at low precision and $b_i(\gamma\Omega)$ at the working precision, we can still deduce $b_i(\Omega)$ at the working precision.

This is what we did in practice. We used the code of Pascal Molin (see \cite{Molin}) for the numerical integration technique and we noticed that, given the roots of a polynomial of degree $6$, it returns
a period matrix of the form $\gamma''\Omega''$, where $\Omega''\in\FF$ and $\gamma''$ always seems to have $-1,0,1$ as coefficients. Thus we never had any problem with the reduction in the fundamental domain of $\gamma''\Omega''$
(recall Assumption~\ref{supp}).

\subsection{Complexity analysis}

Let $f_1$, $f_2$, $f_3$ be three modular functions for a congruence subgroup $\Gamma$ of $\Gamma_2$ generating the function field of $\Gamma$.
Let $p$ be a prime number which is prime to the level of~$\Gamma$ and $C_p$ be a set of representatives of $\Gamma/(\Gamma\cap\Gamma_0(p))$.
We have explained in the preceding section how to find $\Omega$ from $(f_1(\Omega),f_2(\Omega),f_3(\Omega))$ and 
 we have then to evaluate the modular polynomials  (see Definition~\ref{defPolMod}) for a prime $p$ at $\Omega$, which means we have to compute for $\ell=2,3$:
\[\Phi_{1,p}(X,f_1(\Omega),f_2(\Omega),f_3(\Omega))\qquad\textrm{and}\qquad\Psi_{\ell,p}(X,f_1(\Omega),f_2(\Omega),f_3(\Omega)) \] 
(and each coefficient of these polynomials is the evaluation at $\Omega$ of a trivariate rational fraction in $f_1$, $f_2$, $f_3$ that we have to interpolate). 
To do that, we compute first $f_{\ell,p}^\gamma(\Omega)$ for all $\gamma\in C_p$ and $\ell=1,2,3$.  %complexité
%Let $n$ be the maximum of the sum of the total degrees of the numerators and the denominators of the coefficients of the three modular polynomials.
Let $q=p^3+p^2+p+1$ the degree of $\Phi_{1,p}(X)$.
The evaluation of $\Phi_{1,p}(X)$ at $\Omega$ can be obtained in $O(\mathcal{M}(q)\log{q})$
using a subproduct tree (see \cite[Section 10.1]{MCA}). The two other polynomials can be obtained with the same complexity using fast interpolation (see \cite[Section 10.2]{MCA}).

We summarize what we have explained through the following algorithm.
\paragraph{}
\begin{algorithm}[H]%\SetAlgoRefName{AAA}
\KwData{$f_1(\Omega),f_2(\Omega),f_3(\Omega)$, a subgroup $\Gamma$ of $\Gamma_2$ such that $\C_\Gamma=\C(f_1,f_2,f_3)$, a prime $p$ prime to the level of $\Gamma$, a set $C_p$ of representatives of $\Gamma/(\Gamma\cap\Gamma_0(p))$ and the precalculation of the action of $\Gamma_2/\Gamma$ and 
precisions $N$ and $N'$.}
\KwResult{$\Phi_{1,p}(X,f_1(\Omega),f_2(\Omega),f_3(\Omega))$ and $\Psi_{\ell,p}(X,f_1(\Omega),f_2(\Omega),f_3(\Omega))$  at precision $N$ (with some loss) for $\ell=2,3$.} 
\BlankLine
\nl Deduce the $j$-invariants $j_i(\Omega)$ from $f_i(\Omega)$\;
\nl Use Mestre's algorithm to obtain a hyperelliptic curve $Y^2=f(X)$ at precision $N$\;
\nl Deduce the ten $b_i(\Omega)$ at precision $N$ using numerical integration at precision $N'$\;
\nl Invert the functions to find $\Omega'$ with the good $j$-invariants at precicision $N$\;
\nl Compare (permutations and signs) the three $f_i(\Omega)$ with the three $f_i(\Omega')$\;
\nl Deduce  a representative $\gamma$ of this action using the precalculation\;
\nl Compute $\Omega=\gamma\Omega'$\;
\nl Compute the $f_{i,p}^\gamma(\Omega)$ at precision $N$ for all $\gamma\in C_p$\;
\nl Compute $\Phi_{1,p}(X,f_1(\Omega),f_2(\Omega),f_3(\Omega))$ at precision $N$ using a subproduct tree\;
\nl Using fast interpolation, compute $\Psi_{\ell,p}(X,f_1(\Omega),f_2(\Omega),f_3(\Omega))$\;
\caption{Evaluation of the modular polynomials}
\label{algo}
\end{algorithm}
\paragraph{}

The complexity of the algorithm depends on the complexity of the evaluation of the $f_i$ at some $\Omega$. Let $q=p^3+p^2+p+1$.
In the case of the theta constants and functions derived from them (as the $j$-invariants), Steps 1 to 7 are of complexity $O(\mathcal{M'}(N)\log(N))$ (by \cite[Theorem $9.3$]{Dupont}), Step 8 is of complexity $O(q\mathcal{M'}(N)\log(N))$ (by \cite[Theorem $12$]{cmh} under Conjecture \ref{conj}),
Step 9 $O(\mathcal{M}(q)\log(q))$ and Step $10$ $O(\mathcal{M}(q)\log(q))$
 so that the complexity of this algorithm with functions derived from the theta constants is $O(q\mathcal{M'}(N)\log(N)$ $+\mathcal{M}(q)\log(q))\subseteq\tilde{O}(p^3N)$. 
In practice, the limiting step is Step 8 (see Section \ref{impl}).

%\paragraph{}
Suppose $f_1$ is the variable which has the largest degree among all the numerators and denominators of the coefficients of the modular polynomials.
Denote by $d_T^A$ (resp. $d_T^B$)  the maximum of the total degrees of the numerators (resp. denominators) of the coefficients of the three modular polynomials
and by $d_{f_1}$ (resp. $d_{f_2}$, $d_{f_3}$)  the maximum exponent of the variable $f_1$ (resp. $f_2$, $f_3$) appearing in one  of the coefficients of these three polynomials. Let $n=d_T^A+d_T^B\le 6d_{f_1}$.
To obtain the modular polynomials, Algorithm \ref{algo} will be executed $(n+1)(d_{f_2}+1)(d_{f_3}+1)$ times and we will interpolate $3q$ rational fractions.
The complexity to compute the modular polynomials is then \[
(n+1)(d_{f_2}+1)(d_{f_3}+1)\tilde{O}(p^3N)+\tilde{O}(np^3d_{f_2}d_{f_3}N)\subseteq\tilde{O}(d_{f_1}d_{f_2}d_{f_3}p^3N).
\]
Note that we suppose we know the degrees of all the trivariate rational fractions to use the Cauchy interpolation with the extended Euclidean algorithm.
We discuss in Section \ref{impl} how to find these degrees.

%\paragraph{}
As we do not have any explicit bounds on the size of the coefficients of the modular polynomials, we assume that it is sufficient to use a floating point precision of $O(N)$, where $N$ is the size 
of the largest coefficient, to do all the computations so that after the interpolation step it is possible to correctly round the coefficients. We also assume Assumption~\ref{supp}.
\begin{theo}[under Conjecture \ref{conj} and the heuristics of the preceding paragraph]\label{ThPrincipal}
Let $f_1,f_2,f_3$ be three modular functions derived from the theta constants for a congruence subgroup $\Gamma$ of $\,\Gamma_2$ generating the function field $\C_\Gamma$. Let $p$ be a prime number 
prime to the level of~$\Gamma$. 
Then, under the previous assumptions, the modular polynomials for these data can be computed in $\tilde{O}(d_{f_1}d_{f_2}d_{f_3}p^3N)$ time, if the degrees in the $f_i$ of all the coefficients (numerators and denominators) of 
the three modular polynomials are known and if $d_{f_1}=\max(d_{f_1},d_{f_2},d_{f_3})$ and where $N$ is the size of the largest coefficient of these modular polynomials.
\end{theo}

\begin{remm}
\begin{itemize}
\item The conjecture is used for the complexity, but it does not affect the correctness of the algorithm because it is easy  to check whether the period matrix found
is the good one or not at each evaluation step.
\item For the complexity $\tilde{O}(d_{f_1}d_{f_2}d_{f_3}p^3N)$
to be quasi-linear in the output size, we must assume that the average size of the coefficients of the modular polynomials is in $\Omega(N)$. 
\end{itemize}
\end{remm}

\section{Computational results} \label{CompRes}
We present in this section the  modular polynomials we have computed with Streng invariants and with the $b_i'$.
The experimental findings given in this section and proved in the next one are used to optimize the implementation of the computation of the modular polynomials
(see Section \ref{impl}).

\subsection{Modular polynomials with the invariants of Streng}\label{dataStreng}
With the algorithm we have presented in Section \ref{algoRegis}, Régis Dupont \cite{Dupont} has calculated the modular polynomials with Igusa invariants for $p=2$ but because
of the large size of the coefficients and the large degrees of the rational fractions
in $j_1$, $j_2$ and $j_3$, he has calculated only the denominators for $p=3$ and the degrees of the rational fractions.

We begin with some notations to compare the results found between the Igusa and the Streng invariants (see Definitions \ref{defIg} and \ref{defStr}). For $p=2$, the number of isogenies is $p^3+p^2+p+1=15$. Denote for $\ell=2,3$
\[\Phi_{1,2}(X)=X^{15}+\sum_{i=0}^{14}\frac{A_{1,i}(i_1,i_2,i_3)}{B_{1,i}(i_1,i_2,i_3)}X^i\qquad\textrm{and}\qquad
\Psi_{\ell,2}(X)=\sum_{i=0}^{14}\frac{A_{\ell,i}(i_1,i_2,i_3)}{B_{\ell,i}(i_1,i_2,i_3)}X^i.\]
We consider the quotient $A_{j,i}/B_{j,i}$ as the $i$-th coefficient of the $j$-th modular polynomial. The numerator and the denominator of each coefficient are polynomials in $\Z[i_1,i_2,i_3]$.

We recall that Dupont found that the denominators of the three polynomials were of the form $1428j_1^\alpha D_2(j_1,j_2,j_3)^6$ for some integer $\alpha$ ranging between $5$ and $21$ and $D_2$ of degrees $5$, $7$ and $5$ in respectively
$j_1$, $j_2$, and $j_3$ (see \cite[Pages 225--226]{Dupont} for these results and for the definition of $D_2$). With Streng invariants, we have found that the denominators are of the form $ci_3^\alpha D_2'(i_1,i_2,i_3)$ for $\Phi_{1,2}$ and of the form $ci_3^\alpha (D_2'(i_1,i_2,i_3))^2$ for the others, where $c$ is a constant in $\Z$, $\alpha$ varies from $0$ to $3$, and \\

$D_2'=(24576 i_3 i_1^5 + (96i_2^3 - 4608i_3i_2)i_1^4 + (-6220800i_3i_2 - 12288i_3^2)i_1^3 +   
(-23328i_2^4 - 48i_3i_2^3 + 1088640i_3i_2^2 + 2304i_3^2i_2 + 24883200i_3^2)i_1^2 +          
(93312i_3i_2^3 + 419904000i_3i_2^2 - 5909760i_3^2i_2 + (1536i_3^3 - 8398080000i_3^2))i_1 + 
(1417176i_2^5 - 5832i_3i_2^4 + (6i_3^2 - 94478400i_3)i_2^3 + 287712i_3^2i_2^2 +          
(-288i_3^3 + 1154736000i_3^2)i_2 + (-248832i_3^3 + 755827200000i_3^2)))$
\\\phantom{}\\is irreducible.
It is clearl that the exponents of $D_2$ and $D_2'$ are related to the exponent of $h_{10}$ in the definition of the different $j$-invariants.
Denote by $d_{i,j,\ell}$ the degree of the numerator of the $\ell$-th coefficient of the $i$-th modular polynomial  in $i_j$ (see Definition~\ref{defPolMod}) and $\alpha_{i,\ell}$ the exponent of $j_3$ appearing in the denominator of the $\ell$-th coefficient 
of the $i$-th polynomial. The degrees found  are written in Table \ref{degStr2}. 

\begin{table}[h]
\centering
\begin{tabular}{|c|cccc|cccc|cccc|}
\hline
$\ell$&$d_{1,1,\ell}$&$d_{1,2,\ell}$&$d_{1,3,\ell}$&$\alpha_{1,\ell}$&$d_{2,1,\ell}$&$d_{2,2,\ell}$&$d_{2,3,\ell}$&$\alpha_{2,\ell}$&$d_{3,1,\ell}$&$d_{3,2,\ell}$&$d_{3,3,\ell}$&$\alpha_{3,\ell}$\\
\hline
0&25&11&11&3&30&17&15&3&33&17&16&3\\
1&23&11&11&3&28&17&15&3&31&17&16&3\\
2&23&11&11&3&28&17&15&3&31&17&16&3\\
3&21&11&11&3&26&17&15&3&29&17&16&3\\
4&21&11&11&3&26&17&15&3&29&17&16&3\\
5&20&11&10&3&25&17&14&3&28&17&15&3\\
6&20&11&10&3&25&17&14&3&28&17&15&3\\
7&18&10&9&2&23&17&14&3&26&17&15&3\\
8&18&10&9&2&23&16&13&2&26&16&14&2\\
9&16&10&8&2&21&15&12&2&24&15&13&2\\
10&16&8&7&1&21&15&12&2&24&15&13&2\\
11&15&8&7&1&20&13&11&1&23&13&12&1\\
12&15&7&7&1&20&13&11&1&23&13&12&1\\
13&11&6&5&0&16&12&10&1&20&12&11&1\\
14&8&5&4&0&13&11&8&0&16&11&9&0\\
\hline
\end{tabular}
\caption{Degrees of the numerators of the modular polynomials with Streng invariants for $p=2$}
\label{degStr2}
\end{table}

The degrees of the numerators of the coefficients  of the modular polynomials found by Dupont for the Igusa invariants vary from $37$ to $60$ in $j_1$, from $50$ to $75$ in $j_2$ and 
from $33$ to $50$ in $j_3$  for $\Phi_{1,2}(X)$ while they do not exceed $25$ with Streng invariants. The size of the integers in the former case is bounded by $210$ decimal digits and by $105$ in the latter case.
Moreover, the three polynomials computed by Dupont (and accessible at his website) fill $57$ MB and the others $2.1$ MB. 
Thus the Streng invariants provide smaller modular polynomials in terms of degree, precision and total space.  

%chiffres décimaux Regis p=2: phi1 210, psi2 210, psi3 210
%chiffres décimaux Streng p=2: phi1 90, psi2 100, psi3 105
%chiffres décimaux Streng p=3: phi1 391, psi2 551, psi3 

%\paragraph{}
We also managed to compute the three modular polynomials with Streng invariants for $p=3$. The number of isogenies is $40$.
The denominators have the same properties as described before: they are of the form $ci_3^\alpha (D'_3(i_1,i_2,i_3))^2$ for  $\Phi_{1,3}$ and of the form $ci_3^\alpha (D_3'(i_1,i_2,i_3))^4$
for the others. The common part $D'_3$ is an irreducible polynomial which occurs with degrees $13$, $10$ and $8$ in respectively $i_1$, $i_2$ and $i_3$.
Dupont has found that the denominators with the Igusa invariants are of the form $cj_1^\alpha D_3(j_1,j_2,j_3)^{18}$, where $D_3$ has degrees $14$, $20$ and $13$ in respectively
$j_1$, $j_2$ and $j_3$.
We present some degrees of the numerators in Table \ref{degStr3}.

\begin{table}[h]
\centering
\begin{tabular}{|c|cccc|cccc|cccc|}
\hline 
$\ell$&$d_{1,1,\ell}$&$d_{1,2,\ell}$&$d_{1,3,\ell}$&$\alpha_{1,\ell}$&$d_{2,1,\ell}$&$d_{2,2,\ell}$&$d_{2,3,\ell}$&$\alpha_{2,\ell}$&$d_{3,1,\ell}$&$d_{3,2,\ell}$&$d_{3,3,\ell}$&$\alpha_{3,\ell}$\\
\hline
0&61&32&32&4&87&52&48&4&92&52&49&4\\
1&61&32&31&4&87&52&47&4&92&52&48&4\\
2&61&32&31&4&87&52&47&4&92&52&48&4\\
\vdots&&\vdots&\vdots&&&\vdots&\vdots&&&\vdots&\vdots& \\
37&41&22&21&1&67&43&37&1&72&43&39&1\\
38&36&21&19&0&62&42&36&1&67&42&37&1\\
39&31&20&17&0&57&41&33&0&62&41&35&0\\
\hline
\end{tabular}
\caption{Degrees of the numerators of the modular polynomials with Streng invariants for $p=3$}
\label{degStr3}
\end{table}
The degrees are far smaller than those with the Igusa invariants which range from $243$ to $420$. We do not know the size of the integers of the polynomial with Igusa invariants, but in the case of Streng invariants
we have found that they can reach $550$ decimal digits. The three polynomials fill $890$ MB.

\subsection{Modular polynomials with the $b_i'$}\label{databi}

 We have computed the modular polynomials for $b_i'$ for $p=3,5$ and $7$ (see (\ref{bequib}) for their definition). Note that for $p=2$, these polynomials do not exist because 
$\Gamma(2,4)\cap \Gamma_0(2)=\Gamma(2,4)$.
This time there is only one common denominator $D_p$ for all the coefficients of the three polynomials (there are no constants and no powers of one of the $b'_i$). For example, we have\\

$D_3=1024b'^6_3b'^6_2b'^{10}_1 - ((768b'^8_3 + 1536b'^4_3 - 256)b'^8_2 + 1536b'^8_3b'^4_2 - 256b'^8_3)b'^8_1 +(1024b'^6_3b'^{10}_2 + (1024b'^{10}_3+ 2560b'^6_3 - 512b'^2_3)b'^6_2 - (512b'^6_3 - 64b'^2_3)b'^2_2)b'^6_1 - (1536b'^8_3b'^8_2 + (-416b'^4_3 + 32)b'^4_2 + 32b'^4_3)b'^4_1 - ((512b'^6_3 - 64b'^2_3)b'^6_2 - 64b'^6_3b'^2_2)b'^2_1 +256b'^8_3b'^8_2 - 32b'^4_3b'^4_2 + 1.$\\

For $p=5$ (resp. $p=7$), the denominator occurs with exponents $70$ (resp. $226$) in the three $b_i'$.
These three denominators have interesting properties. They are symmetric, the exponents of the $b'_i$ are always even and there are relations modulo $2$ and $4$ between the exponents of each monomial.
We have also noted similar properties for the numerators. In particular, we have noted that for $p=3$ and $5$, $\Psi_{2,p}(X,b'_1,b'_2,b'_3)=
\Psi_{3,p}(X,b'_1,b'_3,b'_2)$ and $\Phi_{1,p}(X,b'_1,b'_2,b'_3)=\Phi_{1,p}(X,b'_1,b'_3,b'_2)$. 
Moreover, the total degrees for the denominators are $24$, $120$ and $226$, which always seems to be $p^3-p$.
We will prove all this in the next section.

%\paragraph{}
Table \ref{degbi3} shows a few of the degrees for $p=3$. This table can be compared with the results found with the $j$-invariants (see Table \ref{degStr3}). The notation is similar as before. 
\begin{table}[h]
\centering
\begin{tabular}{|c|ccc|ccc|}    
\hline
$\ell$&$d_{1,1,\ell}$&$d_{1,2,\ell}$&$d_{1,3,\ell}$&$d_{2,1,\ell}$&$d_{2,2,\ell}$&$d_{2,3,\ell}$\\
\hline
0 & 40&  10&   10& 37&  13 &  12\\
1 & 37  &12   &12&36 & 15  & 14\\
2 & 38  &14  & 14&37 & 17  & 16\\
3 & 39  &16  & 16& 36&  19 &  18\\
4 & 36  &16  & 16&35 & 19  & 18\\
\vdots&&\vdots&&&\vdots&\\
35 & 21  &16  & 16&22 & 19  & 18\\
36 & 20  &16  & 16& 19&  19 &  18\\
37 & 17  &16  & 16&16 & 17  & 16\\
38 & 14  &14  & 14&15 & 15 &  14\\
39 & 13  &12  & 12&12  &13 &  12\\
\hline
\end{tabular}
\caption{Degrees of the numerators of the modular polynomials with the $b'_i$ for $p=3$}
\label{degbi3}
\end{table}

Table \ref{degbi57} indicates the minimal and maximal degrees each of the $b'_i$ do take for the differents modular polynomials for $p=5$ and $7$.
\paragraph{}
\begin{table}[h]
\centering
\begin{tabular}{|c|c|c|c|}
\hline
min-max of&$b'_1$&$b'_2$&$b'_3$\\
\hline
$\Phi_{1,5}$&75-156&70-92&70-92\\
$\Psi_{2,5}$&72-155&75-97&72-94\\
$\Phi_{1,7}$&233-400&226-272&226-272\\
$\Psi_{2,7}$&230-397&233-279&230-276\\
\hline
\end{tabular}
\caption{Degrees of the numerators of the modular polynomials with the $b'_i$ for $p=5,7$}
\label{degbi57}
\end{table}

The integers have about $10$, $60$ and $190$ decimal digits for respectively $p=3$, $5$ and $7$.
The three polynomials fill $270$ KB for $p=3$ (which is $3000$ times smaller than the total space of the modular polynomials with Streng invariants for $p=3$), and $305$ MB for $p=5$ 
while only the two first fill $29$ GB for $p=7$ (we do not have computed the third because it would have taken too much time and we have assumed that there is the same symmetry as in the cases $p=3$ and $p=5$, so that
the third polynomial can be deduced from the second one).   
Compared to the polynomials found with the invariants of Streng for $p=3$, these invariants produce smaller polynomials in terms of degree, precision and total space.

\section{Analysis of the results}\label{analysis}
\subsection{Humbert surfaces}
In this section we will examine the meaning of the denominators appearing in the different modular polynomials. 
The principal tool we use is the notion of Humbert surface, which have been studied in \cite{Gruenewald}.

Let $\Delta\equiv 0,1\bmod 4$ and $\Delta>0$. We call the Humbert surface $H_\Delta$ of discriminant $\Delta$ the irreducible surface 
of matrices which are equivalent to some $\Omega=\begin{ppsmallmatrix}\Omega_1&\Omega_2\\\Omega_2&\Omega_3\end{ppsmallmatrix}$ in $\Gamma_2\bb\HH_2$ satisfying 
$k\Omega_1+\ell\Omega_2-\Omega_3=0$ where $k$ and $\ell$ are determined uniquely by $\Delta=4k+\ell$ and $\ell\in\{0,1\}$.

These surfaces are of particular interest for us because of the next proposition which states that $\mathcal{L}_{p}=H_{p^2}$. 

\begin{prop} Let $m$ be a positive integer. Then the  Humbert surface $H_{m^2}$ is the moduli space for isomorphism classes of principally polarized 
abelian surfaces which split as a product of two elliptic curves via an isogeny of degree $m^2$.
\end{prop} 
\begin{proof}See \cite[Proposition 2.14]{Gruenewald}.\end{proof}

For each discriminant $\Delta$ there is an irreducible polynomial $L_{\Delta}(j_1,j_2,j_3)$ whose zero set is the Humbert surface of discriminant $\Delta$.
Thus, by Lemma~\ref{DenIg}, $L_{p^2}(j_1,j_2,j_3)$ divides the denominators of the modular polynomials with the Igusa invariants.
The exponent to which $L_{p^2}(j_1,j_2,j_3)$ appears in the denominator seems to depend on the exponent of the $h_{10}$ in the definition
of the $j$-invariants. 
A heuristic reason for the factor $j_1^\alpha$ in the denominator of a coefficient of a modular polynomial
is to compensate for the case where $h_{12}(\Omega)=0$ (recall Definition \ref{defIg}). With Streng invariants, there is a factor $i_3^\alpha$ to
compensate for the case where $h_4(\Omega)=0$ (recall Definition \ref{defStr}). Note that $j_1$ (resp. $i_3$) has the greatest exponent of $h_{12}$ (resp $h_4$) in its definition
among $j_1$, $j_2$, $j_3$ (resp $i_1$, $i_2$, $i_3$).

Moreover, a formula for the degree of these surfaces exists. Let 
\[a_{p^2}:=24\sum_{\substack{x\in\mathbb{Z},\\4|(p^2-x^2)}}\sigma_1\left (\frac{p^2-x^2}4\right )+12p^2-2\]
with $\sigma_1(n)=\sum_{d|n}d$ the  sum of positive divisors function. Then

\begin{theo}The degree of any Humbert surface of discriminant $p^2$ can be obtained by the  formula
\[v(p^2)\deg(H_{p^2})+5=\frac {a_{p^2}}{2}
\qquad\textrm{where}\qquad
v(p^2)=\left\{\begin{array}{ll}1/2&\textrm{if }p=2\\
                                      1&\textrm{otherwise.}\end{array}\right.\]
\end{theo}
\begin{proof}See \cite[Theorem 3.8]{Gruenewald}.\end{proof}

Applying this formula gives $\deg(H_4)=60$ and $\deg(H_{9})=120$. Here, the degree of the surfaces
is the degree of the homogenous form of $L_{\Delta}$ with weight $(4,6,10,12)$ for the functions $(h_4,h_6,h_{10},h_{12})$ (see \cite[Pages 170--172]{Geer}).
We have then substituted the $j$-invariants of the common denominator for $p=2$ and $p=3$ by their definition in terms of the $h_i$ and multiplied by a power of $h_{10}$ to homogenize.
The degree we have found for $p=2$ is  $100$ (resp. $300$) with the invariants of Streng (resp. of Igusa), but there is a factor $h_4^{10}$ (resp. $h_{12}^{20}$) and we have $100-40=60$ (resp. $300-240=60$).
 This factor can be explained by the fact that the $j$-invariants are zero when $h_4=0$  (resp. $h_{12}=0$).
For $p=3$, we have found (for Streng invariants) that the degree is $200$ and there is a factor $h_4^{20}$. We have then $200-80=120$.

%\paragraph{}

%2)degree
We study now what happens for our modular polynomials with the theta constants.
We also have a formula for the degree due to the work of Runge (\cite{Runge}, see also \cite{Gruenewald}) who considered finite covers of $\Gamma_2\bb\HH_2$ for the study of Humbert surfaces
because of the large degrees and coefficients of the polynomial with the $j$-invariants.
Define $\Gamma^*(2,4)$ to be the largest normal subgroup of $\Gamma(2,4)$ which does not contain the matrix $\operatorname{diag}(-1,1,-1,1)$.
The natural projection $\pi:\Gamma^*(2,4)\bb\HH_2\to\Gamma_2\bb\HH_2$ is a finite map.
We say that each component of $\pi^{-1}(H_{p^2})$ in $\Gamma^*(2,4)\bb\HH_2$ is a Humbert component and it is possible to define an order $v'_i(p^2)$
for each irreducible Humbert component $F_{p^2,i}$. Since $\Gamma^*(2,4)$ is normal, these components have the same degree.
Moreover by \cite{Runge}, any irreducible component of the covering of $H_{p^2}$ is given by the zero set of a single irreducible
polynomial.
 
\begin{prop} The degree of any Humbert component $F_{p^2,i}$ in $\Gamma^*_{2,4}\bb\HH_2$ is given by the formula 
\[a_{p^2}=10(1+\deg(F_{p^2,i})).\]
\end{prop}
\begin{proof} See \cite[Proposition 3.9]{Gruenewald}.\end{proof}

\begin{prop} Let $p>2$ be a prime number. The degree of $F_{p^2,i}$ is $p^3-p$.\end{prop}
\begin{proof}
From the degree formula above, we have that $a_{p^2}=10(1+\deg(F_{p^2,i}))$ and from the definition of $a_{p^2}$ it suffices to prove that
$\sum_{x>0}\sigma_1(\frac{p^2-x^2}4)=(5p^3-6p^2-5p+6)/24$. The left-hand side can be rewritten as $\frac 12\sum_{k=1}^p\sigma_1(k)\sigma_1(p-k)$ and the result comes
from the equality $\frac{1}{2i\pi}G'_2(\Omega)=\frac 56G_4(\Omega)-2G_2(\Omega)^2$ of \cite{Zagier} where $G_i$ is the $ith$ Eisenstein series (in genus~$1$).
(Moreover, using the fact that $\sigma_1$ is multiplicative, it can be shown that for all $p>2$, the degree of $F_{4p^2,i}$ is also $p^3-p$, but we do not need this result).
\end {proof}

In our case, we use $\Gamma(2,4)$ and not $\Gamma^*(2,4)$, but we noted that the total degrees found for $p=3, 5$ and $7$ is 
always $p^3-p$.
%theta(tau/2) de poids 1 (et theta(tau) de poids 1/2 !)
The reason for this is that the degree formula depends on the number of Humbert components and the order of some 
isotropy subgroup and these numbers are equal for the groups  
 $\Gamma^*(2,4)$ and $\Gamma(2)$ (see \cite{Gruenewald}) so that it is the case for $\Gamma(2,4)$ because
 $\Gamma^*(2,4)<\Gamma(2,4)<\Gamma(2)$. Thus the degree formula of a component of discriminant $p^2$ for the group $\Gamma(2,4)$ is
 the same as those for the group $\Gamma^*(2,4)$, namely $p^3-p$.
Note that the definition of degree here is the total degree of the polynomial because the $\theta_i(\tau/2)$ are Siegel modular forms of weight $1$ for $\Gamma(2,4)$.

%\paragraph{}
%1) den
Consider this time the locus $\mathcal{L}_p'$ of all the principally polarized abelian surfaces modulo $\Gamma(2,4)$
 that are $(p,p)$-isogenous to a principally polarized abelian surface $\Omega$
which is isogenous to a product of two elliptic curves by the $(2,2)$-isogeny $\Omega\to\Omega/2$ and such that $\theta_0(\Omega/2)=0$ (recall that $b'_i(\Omega):=\theta_i(\Omega/2)/\theta_0(\Omega/2)$ and Proposition~\ref{zerodiag}).
\begin{prop}
The denominators of the modular polynomials for the functions $b'_1,b'_2,b'_3$ are divisible by a polynomial $L'_p$ in $\Q[b'_1,b'_2,b'_3]$ describing the preceding locus.
%This polynomial is a Humbert component (of the map $\Gamma(2,4)\bb\HH_2\to\Gamma_2\bb\HH_2$).
\end{prop}
\begin{proof}
We adapt the proof of lemma 6.2 of \cite{BL}. Let $\Omega\in\Gamma(2,4)\bb\HH_2$ which is $(p,p)$-isogenous to $\Omega'$ such that $\theta_0(\Omega'/2)=0$. Let $c$ be a
coefficient of the polynomial $\Phi_{1,p}$. For some $\gamma\in\Gamma(2,4)/(\Gamma(2,4)\cap\Gamma_0(p))$, $b'_{1,p}(\gamma\Omega)$ is infinite.
The evaluation of $c$ at $\Omega$ is a symmetric expression in the $b'^{\gamma}_{1,p}(\Omega)$'s. Generically, there is no algebraic relation between these values and the evaluation of $c$ at $\Omega$
is therefore infinite. Since the $b'_i(\Omega)$ are finite, the numerator of $c$ is finite. We conclude that the denominator of $c$ must vanish at $\Omega$, which means that $c$ is divisible 
by a polynomial describing the locus. The proof for $\Phi_{\ell,p}$, $\ell=2,3$ proceeds similarly.
\end{proof}

We have noticed that for $p=3,5$ and $7$, the coefficients of the three modular polynomials with the $b'_i$ always have  $L'_p$ has denominator (unlike the case with the $j$-invariants where there
also is a factor $j_1$ or $i_3$, as explained in Section \ref{CompRes}). 
%This could be explained by the fact that the three $b'_i$ have the same denominator, which is $\theta_0$.
This justified the following conjecture, which will be used in the next sections. 
\begin{conj} \label{conjden}
The polynomial $L'_p$ is the denominator of all the coefficient of the three modular polynomials.
\end{conj}

% justifier irréductible

\subsection{Symmetries}

As mentioned above (Section \ref{databi}), we have noticed for $p=3,5$ that
$\Psi_{2,p}(X,b'_1,b'_2,b'_3)=
\Psi_{3,p}(X,b'_1,b'_3,b'_2)$ and  for $p=3,5,7$ that $\Phi_{1,p}(X,b'_1,b'_2,b'_3)=\Phi_{1,p}(X,b'_1,b'_3,b'_2)$. 
These symmetries have the following meaning.
For each variety $\Omega\in\Gamma(2,4)$ and $p\Omega$ having invariants $(b_1'(\Omega)$, $b_2'(\Omega)$, $b_3'(\Omega))$ 
and $(b'_{1,p}(\Omega)$, $b'_{2,p}(\Omega)$, $b'_{3,p}(\Omega))$, there exists a variety with invariants 
$(b_1'(\Omega),b_3'(\Omega),b_2'(\Omega))$ and such that one of its $(p,p)$-isogenous varieties has invariants $(b'_1(p\Omega),b'_3(p\Omega),b'_2(p\Omega))$.

A proof of this can be obtained by looking at the action of some matrices.
Indeed, we have that $\Phi_{1,p}$ is the minimal polynomial of $b_{1,p}$, which means it is the unique polynomial such that for all $\Omega\in\HH_2$, 
$\Phi_{1,p}(x,b'_1(\Omega),b'_2(\Omega),b'_3(\Omega))=0$ \ssi $x=b'_{1,p}(\Omega)$; hence $\Phi_{1,p}(b'_{1,p}(\gamma\Omega),$ $b'_1(\gamma\Omega),$ $b'_2(\gamma\Omega),$ $b'_3(\gamma\Omega)
)=0$ for all $\gamma\in\Gamma_2$. What we are looking for is a matrix that fixes $b'_1$ and $b'_{1,p}$ and interchanges $b'_2$ with $b'_3$ and
$b'_{2,p}$ with $b'_{3,p}$. This action on $\Phi_{1,p}(X)$ would provide a unitary polynomial with the same roots and degree as $\Phi_{1,p}(X)$ and since $\Phi_{1,p}(X)$ is a minimal polynomial, they both have to be equal.

Assume that we have the symmetry for $\Phi_{1,p}$. Then by Definition~\ref{defPolMod} we have $b'_{\ell,p}=\Psi_{\ell,p}(b'_{1,p})/\Phi'_{1,p}(b'_{1,p})$ for $\ell=2,3$.
We use this action on $\Psi_{2,p}(X)$ which gives us  \[b'_{3,p}=\Psi_{2,p}(b'_{1,p},b'_1,b'_3,b'_2)/\Phi_{1,p}'(b'_{1,p},b'_1,b'_2,b'_3),\] so that 
\[\Psi_{2,p}(b'_{1,p},b'_1,b'_3,b'_2)=b'_{3,p}\Phi_{1,p}'(b'_{1,p},b'_1,b'_2,b'_3)=\Psi_{3,p}(b'_{1,p},b'_1,b'_2,b'_3).\]

%On $\Psi_{2,p}(X)$, it produces a polynomial with the same roots as $\Psi_{3,p}(X)$, and of the same degree hence it should be equal to $\Psi_{3,p}(X)$ up to a constant. 

Firstly, the search is done among the representatives of $\Gamma_2/\Gamma(2,4)$ because $\Gamma(2,4)$ fixes the $b'_i$. 
A representative of the unique class such that $(b'^\gamma_1,b'^\gamma_2,b'^\gamma_3)=(b'_1,b'_3,b'_2)$ is
\[\gamma=\begin{ppsmallmatrix}1&-3&-2&2\\0&1&2&0\\0&0&1&0\\0&-4&-5&1\end{ppsmallmatrix}.\]

Secondly, we look for a matrix $\gamma'$ in $\Gamma(2,4)$ such that $\gamma\gamma'\in\Gamma_0(p)$. For $p=3,5$ and $7$ we can take for $\gamma'$ respectively
\[\begin{ppsmallmatrix}-5&24&-12&12\\-2&19&-12&8\\0&6&-5&2\\-2&4&0&3\end{ppsmallmatrix},
\begin{ppsmallmatrix}-7&6&4&2\\0&-7&2&0\\0&10&-3&0\\10&-8&-6&-3\end{ppsmallmatrix}\textrm{ and } 
\begin{ppsmallmatrix}13&12&-16&-6\\-10&-3&10&4\\56&14&-55&-22\\30&-40&-12&-7 \end{ppsmallmatrix}.
\]

%expliquer pourquoi G/G24 et G24/G0p
%donner matrice 8952   la seule classe tel que b1(gamma)=b1, b2(gamma)=b3, b3(gamma)=b2  (perm plus signe =1) 
%prouver pour tout p ou donner pour p=3,5,7 matrice dans G24
%prouver que alors pour on a matrice dans mm classe et donc cqfd

Recall that for a matrix $X$, we denote by $X_0$ the vector composed of the diagonal entries of $X$.
\begin{lemm}\label{lemclasse}  
Let $M=\begin{ppsmallmatrix}A'&B'\\C'&D'\end{ppsmallmatrix}
\in\Gamma_2/\Gamma(2,4)$ and $M'\in\Gamma(2,4)$ such that $MM'\in\Gamma_0(p)$, for some prime $p>2$. Then $(MM')_p$ is in the same equivalence class as $M$
for all $p\equiv 1\bmod 4$. For $p\equiv 3\bmod 4$, this is the case if we have the additional properties $(A'\,^t\!B')_0\equiv 0\bmod 2$ and $(C'\,^t\!D')_0\equiv 0\bmod 2$.
\end{lemm}
\begin{proof}
Let $MM'=\pMat$. We study under which conditions  $(MM')_pM^{-1}\in\Gamma(2,4)$ or, equivalently, when $(MM')_p(MM')^{-1}$ is in $\Gamma(2,4)$. We have 
\[\begin{ppsmallmatrix}A&pB\\C/p&D\end{ppsmallmatrix} \begin{ppsmallmatrix}\,^t\!D&-\,^t\!B\\-\,^t\!C&\,^t\!A\end{ppsmallmatrix}=
\begin{ppsmallmatrix}A\,^t\!D-pB\,^tC&-A\,^t\!B+pB\,^t\!A\\C/p\,^t\!D-D\,^t\!C&-C/p\,^t\!B+D\,^t\!A\end{ppsmallmatrix},
\]
where $\begin{ppsmallmatrix}\,^t\!D&-\,^t\!B\\-\,^t\!C&\,^t\!A\end{ppsmallmatrix}$ is the inverse of $M$ by Equation~(\ref{eqSpg}).
As $p\equiv 1\bmod 2$, this product is the identity modulo $2$.
Now for $p\equiv 1\bmod 4$, we have $-A\,^t\!B+pB\,^t\!A\equiv C/p\,^t\!D-D\,^t\!C \equiv 0\bmod 4$ (recall that this product is in $\Gamma_2$) so that
$(MM')_p(MM')^{-1}\in\Gamma(2,4)$. For $p\equiv 3\bmod 4$, we have $-A\,^t\!B+pB\,^t\!A\equiv 2A\,^t\!B\bmod 4$ and $C/p\,^t\!D-D\,^t\!C\equiv 2C\,^t\!D\bmod 4$.
Thus to be in $\Gamma(2,4)$, we want $(A\,^t\!B)_0\equiv (C\,^t\!D)_0\equiv 0\bmod 2$. Finally note that  $M'\equiv I_4\bmod 2$ and we deduce  the lemma.
\end{proof}

By this lemma, we have that $(\gamma\gamma')_p$ is in the same equivalence class as $\gamma$ for any prime $p>2$, hence the permutation $(b'^{\gamma\gamma'}_{1,p},b'^{\gamma\gamma'}_{2,p},b'^{\gamma\gamma'}_{3,p})=(b'_{1,p},b'_{3,p},b'_{2,p})$.
Moreover the surjectivity of $\textrm{Sp}_4(\Z)\to\textrm{Sp}_4(\Z/4p\Z)$ and the Chinese remainder theorem prove that the matrix $\gamma'$ always exists. Thus there are these symmetries 
for all prime $p>2$ (see Theorem~\ref{ThSym}).
%\paragraph{}

By the above we have also proved that the denominator is always symmetric in $b'_2$ and $b'_3$.
To prove that $L'_p$ is also symmetric in $b'_1$ and $b'_2$ (resp. $b'_1$ and $b'_3$), we use the matrices 
\[
\gamma_{410}=
\begin{ppsmallmatrix}0&-1&0&0\\-1&0&0&0\\0&0&0&-1\\0&0&-1&0\end{ppsmallmatrix}\qquad\textrm{ and }\qquad
\gamma_{8316}=\begin{ppsmallmatrix}1&0&0&2\\-3&1&2&-2\\-4&0&1&-5\\0&0&0&1\end{ppsmallmatrix}
\]
which fixes $b'_3$ (resp. $b'_2$) and interchanges $b'_1$ with $b'_2$ (resp. with $b'_3$).

The action of $\gamma_{410}$ (resp. $\gamma_{8316}$) on $L'_p$ provides an irreducible polynomial and with the same roots as $L'_p$, which are still in $\mathcal{L}_p'$ by the following lemma.
Hence this polynomial is $L'_p$ and thus it is symmetric.

\begin{lemm}\label{memeHumbert}
%Let $\Omega\in\mathcal{L}_p'$ and  $\gamma=\pMat\in\Gamma_2/\Gamma(2,4)$ satisfying the additional properties of lemma \ref{lemclasse} and such that 
%its action on the theta constants send $\{0,4,8,12\}$ to itself. Then $\gamma\Omega$ is in $\mathcal{L}_p'$.
Let $\Omega\in\mathcal{L}_p'$, $\gamma=\pMat\in\Gamma_2/\Gamma(2,4)$ and $\gamma'$ such that $\gamma_p$ is in the same equivalence class of $\gamma'$.
Suppose that the action of $\gamma'$ on the theta constants sends $\{0,4,8,12\}$ to itself. Then $\gamma\Omega$ is in $\mathcal{L}_p'$.
\end{lemm}
\begin{proof}
For $\Omega\in\Gamma(2,4)\bb\HH_2$ to be in $\mathcal{L}_p'$ means that there exists $M\in\Gamma (2,4)/(\Gamma (2,4)\cap\Gamma_0(p))$ satisfying $\theta_0(pM\Omega/2)=0$.
Let $M'\in\Gamma (2,4)/(\Gamma (2,4)\cap\Gamma_0(p))$ be such that $(M'\gamma)M^{-1}\in\Gamma_0(p)$. There exists then $\gamma''\in\Gamma_0(p)$ with $M'\gamma=\gamma''M$. We have, using the 
duplication formula (Proposition~\ref{dupl})
\[
\theta_0(pM'\gamma\Omega/2)=\theta_0(p\gamma''M\Omega/2)=\theta_0(\gamma''_p(pM\Omega)/2)=\sum_{i\in\{0,4,8,12\}}\theta_i^2(\gamma''_p(pM\Omega)).
\]
Moreover,  $\gamma''_p=(M'\gamma M^{-1})_p$ is in the same equivalence class as $\gamma_p$, namely $\gamma'$ by hypothesis (recall that $\Gamma(2,4)$ is a normal subgroup). 
%We use the lemma \ref{lemclasse} on $M'\gamma$ and $M^{-1}$ to conclude that $\gamma'_p=(M'\gamma M^{-1})_p$ is equivalent to $(M'\gamma)$ which is equivalent to $\gamma$ modulo $\Gamma(2,4)$.
The action of $\gamma'$  sends $\{0,4,8,12\}$ to itself, so that
\[\sum_{i\in\{0,4,8,12\}}\theta_i^2(\gamma''_p(pM\Omega))=\zeta_{\gamma''_p}^2\det(...)\sum_{i\in\{0,4,8,12\}}\theta_i^2(pM\Omega)
=\zeta_{\gamma''_p}^2\det(...)\theta_0(pM\Omega/2)=0.\]
\end{proof}

We have proved
\begin{theo}\label{ThSym}
Let $p>2$ be a prime number. The modular polynomials for $b'_1,b'_2$ and $b'_3$ satisfy
\[
\Phi_{1,p}(X,b'_1,b'_2,b'_3)=\Phi_{1,p}(X,b'_1,b'_3,b'_2)\qquad\textrm{and}\qquad\Psi_{2,p}(X,b'_1,b'_2,b'_3)=\Psi_{3,p}(X,b'_1,b'_3,b'_2).
\]
Morover, the polynomial $L'_p$ is symmetric.
\end{theo}

%Another interpretation of these symmetries is that there is the following equality:
%\[\{(b'_1(pM\gamma\Omega),b'_2(pM\gamma\Omega),b'_3(pM\gamma\Omega)),M\in\Gamma(2,4)/(\Gamma_0(p)\cap\Gamma(2,4))\}=\{(b'_1(pM\Omega),b'_3(pM\Omega),b'_2(pM\Omega))\}\]
%where  $\gamma="\gamma\gamma'"$   sans qu'il y ait égalité entre $b'_1(pM\gamma\Omega)$ et $b'_1(pM\Omega)$ par exemple
\subsection{Relations modulo 2 and 4}

We study now the different relations modulo $2$ and $4$ between the exponents of the $b'_i$ in each coefficient.
Consider the numerator
of the $\ell$-th coefficient of the $m$-th modular polynomial for $m=1$ or $2$ (we have seen that the third polynomial can be deduced from the second one), whose
monomials are of the form $c_{ijk}b'^i_1b'^j_2b'^k_3$. We have found that for $p=3,5$ and $7$, if $c_{ijk}\ne 0$, then
\begin{equation}\label{rel24}
\begin{split}
&i   \equiv \ell+m+1  \bmod 2  \\
&i+j \equiv -p\ell    \bmod 4  \\
&j+k \equiv p(m-1) \bmod 4  \\
\end{split}
\end{equation} 
and with similar notation, we always have 
\begin{equation}\label{rel24Den} i\equiv j\equiv k\equiv 0\bmod 2 \qquad\textrm{and}\qquad i+j\equiv j+k\equiv 0\bmod 4 \end{equation}
for the denominators.
These equalities are determined by the existence of some matrices $\gamma$ with the property that $b'_i(\gamma\Omega)=\imath^{\alpha_i} b'_i(\Omega)$ and $b'_{i,p}(\gamma\Omega)=\imath^{\beta_i} b'_{i,p}(\Omega)$ with $\alpha_i$ and $\beta_i$ in $\{0,1,2,3\}$.
We will denote the action of such matrices by the vector  $(\alpha_1,\alpha_2,\alpha_3,\beta_1,\beta_2,\beta_3)$.

With the same arguments as before, we deduce that such an action produces a polynomial with the same roots and degrees as $\Phi_{1,p}(X)$ (resp. $\Psi_{2,p}(X)$) which is then $\Phi_{1,p}(X)$ (resp. $\Psi_{2,p}(X)$)
up to a constant. As $p^3+p^2+p+1\equiv 0\bmod 4$ for any prime $p>2$ and as the leading coefficient of $\Phi_{1,p}(X)$ is $X^{p^3+p^2+p+1}$, we conclude that such an action does not change $\Phi_{1,p}(X)$.
This is not the case of $\Psi_{2,p}(X)$ which is of degree $p^3+p^2+p$ in $X$. 

%\paragraph{}
The matrix \[\gamma_{134}=\begin{ppsmallmatrix}-1&0&0&0\\0&-1&0&0\\2&1&-1&0\\1&0&0&-1\end{ppsmallmatrix}\] acts by $(-1,1,1,-1,1,1)$ for all $p$ by the functional equation of Proposition~\ref{funceq} and Lemma~\ref{lemclasse}. 

Using lemma~\ref{memeHumbert} shows that this matrix preserves the Humbert component (and $(\gamma_{134})_p$ is in the same equivalence class of $\gamma_{134}$ by Lemma~\ref{lemclasse}). Thus we obtain a polynomial with the same roots and degrees as $L'_p$: it is a multiple of $L'_p$.
As the latter is irreducible, it contains at least one monomial where there is not $b'_1$ so that the matrix does not  change this monomial and the constant is thus $1$.
As $L'_p$ is symmetric, we deduce that it has even exponents in $b'_1$, $b'_2$ and $b'_3$.

If we assume Conjecture \ref{conjden}, then we just proved that the action of $\gamma_{134}$ on the numerators do not depend on its action on the denominator. 
Then, on the numerators of $\Phi_{1,p}$, the action of $\gamma_{134}$ shows that $i+\ell$ is always even.
For $\Psi_{2,p}(X)$, we have to determine the constant which appears.
The  leading coefficient of this polynomial is $\sum_{\gamma\in C_p}b'^{\gamma}_{2,p}X^{p^3+p^2+p}$. Consider now the minimal polynomial $\prod_{\gamma\in C_p}(X-b'^\gamma_{2,p})$ of $b'_2$
and note that it is invariant by the preceding action, which is thus also the case of $\sum_{\gamma\in C_p}b'^{\gamma}_{2,p}$. We deduce that the constant is $-1$ (because of the $X^{p^3+p^2+p}$), namely
\[\sum_{\gamma\in C_p}b'^{\gamma\gamma_{134}}_{2,p}(b'^{\gamma_{134}}_{1,p})^{p^3+p^2+p}=-
\sum_{\gamma\in C_p}b'^{\gamma}_{2,p}(b'_{1,p})^{p^3+p^2+p}.\]
We have thus shown the first of the three equalities of (\ref{rel24}).
%\paragraph{}

For the other two, we have to consider the matrices 
\[\gamma_{141}=\begin{ppsmallmatrix}-1&0&0&0\\ 0&-1&0&0\\1&1&-1&0\\1&1&0&-1 \end{ppsmallmatrix}\quad\textrm{ and } \quad
\gamma_{21}=\begin{ppsmallmatrix} -1&0&0&0\\0&-1&0&0\\0&0&-1&0\\0&1&0&-1 \end{ppsmallmatrix}.\] 
Their action for $p\equiv 1\bmod 4$ are respectively $(\imath,\imath,1,\imath,\imath,1)$ and $(1,\imath,\imath,1,\imath,\imath)$ and 
for $p\equiv 3\bmod 4$
it is $(\imath,\imath,1,-\imath,-\imath,1)$ and $(1,\imath,\imath,1,-\imath,-\imath)$ because in 
this case $(\gamma_{141})_p$ and $(\gamma_{21})_p$ are equivalent to  %% preuve ?
\[\gamma_{1886}=\begin{ppsmallmatrix} -1&0&0&0\\0&-1&0&0\\-1&1&-1&0\\1&-1&0&-1 \end{ppsmallmatrix}\quad\textrm{ and }\quad
\gamma_{155}=\begin{ppsmallmatrix} -1&0&0&0\\0&-1&0&0\\0&0&-1&0\\0&3&0&-1 \end{ppsmallmatrix}.\]

On $L'_p$, the action of $\gamma_{141}$ does not change the Humbert component by Lemma~\ref{memeHumbert},
so that $L'_p(\imath b'_1,\imath b'_2,b'_3)$ is a multiple of $L'_p$. 

As $L'_p$ is irreducible, there is a monomial without $b'_1$, which
is then of the form $cb'^i_2b'^j_3$ for some constant $c$. We have already shown that $i\equiv j\equiv 0\bmod 2$ so that $c(\imath b'_2)^ib'^j_3=\pm cb'^i_2b'^j_3$. If it is equal,
then the action of $\gamma_{141}$ fixes $L'_p$. Otherwise $i\equiv 2\bmod 4$ and as $L'_p$ is symmetric, we also have the monomials $cb'^i_1b'^j_2$ and $cb'^i_3b'^j_2$.
Now look at the latter: $cb'^i_3(\imath b'_2)^j=\pm cb'^i_3b'^j_2$. If it is not equal, then $j\equiv 2\bmod 4$ and then $c(\imath b'_1)^i(\imath b'_2)^j=cb'^i_1b'^j_2$. In all cases, 
the action of $\gamma_{141}$ fixes $L'_p$. We can adapt this proof on $\gamma_{121}$ and deduce (\ref{rel24Den}).

We use similar arguments on $\Phi_{1,p}(X)$ and $\Psi_{\ell,p}(X)$ to prove (\ref{rel24}).
Thus we obtain
\begin{theo}\label{ThDeg}
Let $p>2$ be a prime number. Then the polynomial $L'_p$ satisfy (\ref{rel24Den}). Moreover, if we assume the Conjecture~\ref{conjden}, the numerators of the two first modular polynomials
verifies (\ref{rel24}).
\end{theo}

\section{Implementation}\label{impl}
\subsection{External packages}
Dupont presented two algorithms to compute theta functions. 
The first  one uses the definition as sums of exponentials and it computes $\theta_i(\Omega)$ for $i=0,1,2,3$, $\Omega\in\FF$
at precision $N$ with a complexity of $O(\mathcal{M'}(N)N)$.  %, where $\mathcal{M}(N)$ is the time complexity of multiplying two numbers of $N$ bits.  %prop3 cm
The second one uses  Newton lifts and the Borchardt mean and is in $O(\mathcal{M'}(N)\log(N))$ under Conjecture~\ref{conj}. It computes
$\theta_i^2(\Omega)/\theta_0^2(\Omega)$, $i=1,2,3$.
These algorithms have been studied and implemented by Enge and Thomé in \cite{cmh,cmhlogiciel}. Using finite differences, they proved that the complexity to compute the squares of the theta constants
is in $O(\mathcal{M'}(N)\log(N))$ under conjecture \ref{conj}.

We used the cmh library written in C for the evaluation of the square of the theta functions (we also recovered from it the implementation of Mestre's algorithm
and some other functions that were already written in GP) and we used the pari-gnump software \cite{parignump} for switching between number types from the GNU 
multiprecision ecosystem (GMP, MPFR and MPC \cite{gmp,mpfr,mpc}) and corresponding types in Pari/GP to be able to use the algorithm of cmh with GP.

There are two reasons for which the algorithms to compute the theta constants are defined for $\Omega$ only in the fundamental domain. The first one is for the convergence and the second
is because we can use the functional equation of Proposition~\ref{funceq} to obtain the theta constants at $\Omega\in\HH_2$ from the theta constants at $\Omega'\in\mathcal{F}_2$.
We have implemented an algorithm to compute the squares of the theta constants for any matrix in $\HH_2$ with GP \cite{Pari}.

%\paragraph{}

For Algorithm \ref{algo}, we need a method to reduce some $\Omega\in\HH_2$ into the fundamental domain.
 We implemented the standard method (see \cite{Gott,Dupont}).
We also used the code of Pascal Molin \cite{Molin} to compute $\Omega\in\HH_2$ corresponding to a given hyperelliptic curve equation.

Moreover, we have to know the cosets of $\Gamma(2,4)/(\Gamma_0(p)\cap\Gamma(2,4))$ for some primes $p$.
They are naturally calculated beforehand.
A generalization of Algorithm $2$ of \cite{Dupont} to dimension $2$ allows one to compute, 
for subgroups $\Gamma'\subset\Gamma$ of $\Gamma_2$,
 the representatives of the classes of $\Gamma/\Gamma'$ and a set of generators of $\Gamma'$ from
a set of generators of $\Gamma$ and from a function which decides if a matrix lies in $\Gamma'$ or not. 
We apply it twice: first on $\Gamma=\Gamma_2$ and $\Gamma'=\Gamma(2,4)$, then
on $\Gamma=\Gamma(2,4)$ and $\Gamma'=\Gamma_0(p)\cap\Gamma(2,4)$. Another solution consists in using Proposition 10.1 of \cite{Dupont} which provides a set 
of representatives of $\Gamma_2/\Gamma_0(p)$ for all $p\ge 2$. 
We have to multiply each representative by a matrix in $\Gamma_0(p)$ such that the resulting matrix is in $\Gamma(2,4)$, which is
possible by the Chinese remainder theorem.

%\paragraph{}

\subsection{Evaluation and interpolation}
Until now we have presented the algorithm from a theoretical point of view.
In practice, we proceed as follows. 
Since we want to use fast interpolation, it is necessary to know the degrees of the coefficients in the three invariants
$f_1$, $f_2$ and $f_3$.
For example, let $F(f_1,f_2,f_3)$ be one of the coefficients we want to compute.
To obtain the total degree of the numerator and of the denominator of $F$,
it is enough to compute the matrices $\Omega$ in the Siegel space with Algorithm~\ref{algo} 
such that $(f_1(\Omega)$, $f_2(\Omega)$, $f_3(\Omega))=$ $(x_i,x_iy,x_iz)$ for some $x_i$ and fixed $y$ and $z$, to evaluate
$F(x_i,x_iy,x_iz)$  and then to do the interpolation of a univariate rational fraction. 
This also gives upper bounds for the degrees in $f_1$,
$f_2$ and $f_3$. To obtain the degrees in $f_1$ (and similarly in the others), we can compute $F(x_i,y,z)$ and interpolate,
but this will not give the a correct answer every time (even if we assume that the precision is correct and that we have enough $x_i$).
 Indeed,  some simplifications may occur.
Thus, to be sure of the result, it is preferable to evaluate and interpolate for many values of $y$ and $z$ and also for $F(X+r,y+s,z+t)$ for some values of $r$, $s$ and $t$.

Once we have this information, we have two choices for how to proceed.
The first consists in doing sufficiently many evaluations to compute all the coefficients (in $X$) of the three modular polynomials with  interpolation 
of rational fractions. An evaluation means the computation of the modular polynomials at $\Omega$ such that
 $(f_1(\Omega),f_2(\Omega),f_3(\Omega))$ is of the form $(x_i,x_iy_j,x_iz_k)$. 
Otherwise we focus first on only one coefficient (the one with the lowest total degree)
to compute the common denominator and then we do sufficiently many evaluations (here of the form $(x_i,y_j,z_k)$) to compute the 
other coefficients using interpolations of multivariate polynomials. We can speak about polynomials because we can multiply
each evaluation by the evaluation of the denominator (and in the case of the Streng invariants, also by an exponent of $i_3$).

In the first case, the number of evaluations will depend on the maximal total degree  of the three polynomials, while in the
second case, the total degree will intervene only for the coefficient with lowest degrees. Moreover, the precision needed to
interpolate rational fractions is greater than those to interpolate polynomials (and the complexity of an evaluation of the
modular polynomials at some matrices of $\HH_2$ depends on the precision) and it is easier to interpolate polynomials
than rational fractions. For the second choice, the degree tables suggest focusing on the coefficient of highest degree (in $X$) of $\Phi_{1,p}(X)$.

%One can choose to take integer values for the invariants but the precision needed will be too big and the time of evaluation too.
%It is preferable to take floating point values and eventually use rational reconstruction once the polynomials have been computed
%to have the exact coefficients.

One can choose to take integer values for the invariants. The matrix $\Omega$ with these invariants and also the invariants of the isogenous varieties will
not take integer values, but each coefficient of the evaluated modular polynomials will be a rational number. Thus it could be possible at each evaluation to find these rational numbers
using continued fractions (if the working precision is good enough). The interpolation phase could then be done using exact values. 
However when doing this, the precision needed in practice will increase and the time of evaluation too. 
It is preferable to take floating point values for the invariants and reconstruct the rational numbers once the polynomials 
have been interpolated at the working precision to find the exact coefficients.

\subsection{Timings}
Note that in the evaluation there are two steps: given $(f_1(\Omega),f_2(\Omega),f_3(\Omega))$ find $\Omega$ and then
evaluate the modular polynomials at $\Omega$. The last one takes most of the time  (at large enough precision). For  example for $p=5$ 
and $7$ at precision $1000$ decimal digits it takes $0.5$ seconds to compute $\Omega$ from the $b'_i(\Omega)$ and
 the computation of the two polynomials $\Phi_{1,p}(X,b'_1(\Omega),b'_2(\Omega),b'_3(\Omega))$ and 
$\Psi_{2,p}(X,b'_1(\Omega),b'_2(\Omega),b'_3(\Omega))$ take  $12$ and $30$ seconds for respectively $p=5$ and $p=7$ (this difference is due to the number of isogenies: $156$ for one and $400$ for the other). 
% préciser:   calcul: nb eval den +nb eval num contre tout d'un coup 
% mais interpolation num peut être précision plus petite 
%\paragraph{}

%This number for $p=11$ is $1464$ so that the computation for this level seems very difficult. 

We focus now on the computation of the modular polynomials with Streng invariants (recall the results of Section \ref{dataStreng}). 
We proceed with the second method which is not always faster (because it requires two evaluation steps), but it has the advantage of providing the denominator
which is the origin of most of the difficulties when computing modular polynomials. Moreover, we do the interpolation of univariate rational fraction with linear algebra because it is fast enough. 

In level $2$, the largest total degree, of the numerator of the coefficient of degree $14$ of $\Phi_{1,2}(X)$ is $9$ 
and that of the denominator $D'_2$ is $7$. To compute the denominator it is enough to do $(9+7+2)(5+1)(4+1)=540$ evaluations.
Once we have computed them, we do $(33+1)(17+1)(16+1)=10404$ evaluations to compute the numerators (see Table \ref{degStr2}). All of this can be done 
at a precision of $100$ decimal digits. An evaluation takes around $1.33$ second so that the denominator can be computed
in around $12$ minutes and all the polynomials in $4$ hours (on one processor).

In level $3$, the total degrees are $35$ for both the numerator (of the coefficient of degree $39$ of $\Phi_{1,3}$) and the denominator.
The denominator can be computed with $(35+35+2)(20+1)(17+1)=27216$ evaluations in $17$ hours at precision $300$ and then
all the numerators with $(92+1)(52+1)(49+1)=246450$ evaluations (see Table \ref{degStr3}) in around $30$ days at precision $1000$. (The difference
in precision here comes from the fact that the integers of the denominator are much smaller than the integers of the numerators).
The interpolation phase takes around $1$ hour.

%\paragraph{}

To compute the modular polynomials with the $b'_i$, we can use the results found in Sections \ref{databi} and \ref{analysis}. 
In particular, we only have to compute the first two modular polynomials.

For $p=3$, the total degrees are $25$ and $24$ for the numerator and the denominator of the $39$-th coefficient.
It takes around $(25+24+2)(12+1)(12+1)/32\approx 270$ evaluations to obtain the denominator and  around 
$(40+1)(19+1)(18+1)/32 \approx 487$ for the numerators (see Table \ref{degbi3}).
We used $100$ decimal digits for the precision and then an evaluation takes approximately  $0.6$ seconds so that the (two and thus the three) modular polynomials 
can be obtained in less than $10$ minutes (the interpolation phase is negligible). 

For $p=5$, the total degrees are $121$ and $120$ for the numerator and the denominator of the $155$-th coefficient.
The theoretical numbers of evaluations for the denominator and the numerators are  $(121+120+2)(72+1)(72+1)/32< 40500$ and $(156+1)(97+1)(94+1)/32< 46000$ (see Table \ref{degbi57}).
They can be done at precision $1000$ decimal digits where each evaluation takes roughly $12$ seconds.
The polynomials can be calculated in less than $12$ days (on one processor). The interpolation can be done in less than $2$ hours.

For $p=7$, we have computed at first the common denominator because of memory space (the two first polynomials fill $29$ GB). 
Moreover we found that the leading coefficient of the denominators in $b'_1$ is respectively $2^{10}b'^6_2b'^6_3b'^{10}_1$ and
$2^{70}b'^{10}_2b'^{10}_3b'^{70}_1$, so that we conjectured it would be of the same kind for $p=7$. Through some experimentations, we found that it was $2^{226}b'^{38}_2b'^{38}_3b'^{226}_1$.
Knowing this monomial allows one to interpolate as explained in the second paragraph after Remark \ref{simplpol}, which reduces the number of evaluations because this number depends on the degree in $b'_1$
instead of the total degree.

%because  this number will not depend on the total degree of the coefficients of the polynomials (see the paragraph after the first example of the section interpolation).

The degrees of the $399$-th coefficient are $233$ and $226$ (and the total degrees are $337$ and $336$ so that the gain is significant).
The number of evaluations for the denominator was around $(233+226+2)(226+1)(226+1)/32< 727000$ and for 
the numerators of the two modular polynomials around $(400+1)(279+1)(276+1)/32< 972000$ (see Table \ref{degbi57}).
 For the denominator, we managed to compute it in less than $700$ days at precision $2000$ and for the numerators in 
around $2000$ days at precision $3000$. The interpolation time was around a week. It is negligible compared to the evaluation time.

%\paragraph{}
Finally note that each evaluation is independent of the others so that the computation of modular polynomials
is highly parallelizable. The interpolation of a coefficient is independent of the interpolation of the others so that the interpolation
step is also parallelizable. Moreover, it is possible to parallelize the interpolation of a single coefficient.
%In practice, we have parallelized the 

\section{Examples of isogenous curves}

% comme algo dépend heuristique, preuve ?
% preuve à travers application
% autre arg 
The main purpose of the modular polynomials is to find hyperelliptic curves with isogenous Jacobians, in particular over a finite field.
We give some examples with the different polynomials we have computed. Note that the algorithm we have presented is heuristic 
because we have no bounds on the precision loss and we have no proof that the polynomials we found are correct.
We could do interval arithmetic; what we do instead is to heuristically check  for correctness on additional random values not yet used during the evaluation/interpolation algorithm:
for some $\Omega\in\HH_2$, we have to verify that 
\[\Phi_{1,p}(f_{1,p}(\Omega),f_1(\Omega) ,f_2(\Omega) ,f_3(\Omega))=0\] and that for $\ell=2,3$
\[f_{\ell,p}(\Omega)=\Psi_{\ell,p}(f_{1,p}(\Omega),f_1(\Omega) ,f_2(\Omega) ,f_3(\Omega))/\Phi'_{1,p}(f_{1,p}(\Omega),f_1(\Omega) ,f_2(\Omega) ,f_3(\Omega)).\] 
%We have tested several values which give us a high confidence on all the modular polynomials  we have computed.
With one high precision computation, one can be virtually certain that the result is correct.

%\paragraph{}
The Jacobians of the following curves are $(3,3)$-isogenous varieties. We computed the curves using the modular polynomials with Streng invariants.
The first ones over $\F_{5261}$:
\begin{flushleft}\begin{tabular}{lll}
   $Y^2$&$=$&$272X^5 + 4278X^4 + 4297X^3 + 4063X^2 + 1069X + 2998$,\\$Y^2$&$=$&$695X^5 + 2322X^4 + 3115X^3 + 4588X^2 + 1453X + 655$
  \end{tabular}\end{flushleft}
and the following ones over  $\F_{2534267893}$:
  \begin{flushleft}\begin{tabular}{lll}
   $Y^2$&$=$&$1774507961X^6 + 48872812X^5 + 2028583210X^4 + 1092030439X^3 +$\\ &&$ 671225738X^2 + 2233670825X + 608155867$,\\$Y^2$&$=$&$1927466494X^6 + 2286039407X^5 + 1720123333X^4 + 87910848X^3 + $\\ &&$2422852850X^2 + 183139891X + 825611194$.
  \end{tabular}\end{flushleft}
%We also give an example of curves with $(3,3)$-isogenous Jacobians computed using the modular polynomials with the $b'_i$ over $\F_{1073741831}$:
%\begin{flushleft}\begin{tabular}{lll}
%   $Y^2$&$=$&$759384538X^5 + 156120439X^4 + 148044005X^3 + 945677612X^2 + 637622919X +$\\ &&$  1062083212$,\\$Y^2$&$=$&$978185514X^5 + 735378849X^4 + 15113568X^3 + 912622654X^2 + 844620888X +$\\ &&$  989235196$
%  \end{tabular}\end{flushleft}
We also give two examples of curves with $(5,5)$-isogenous Jacobians computed using the modular polynomials with the $b'_i$. Over $\F_{101}$:

\begin{flushleft}\begin{tabular}{lll}
   $Y^2$&$=$&$27X^5 + 71X^4 + 91X^3 + 59X^2 + 5X + 14$,\\$Y^2$&$=$&$29X^5 + 26X^4 + 38X^3 + 20X^2 + 7X + 51$
  \end{tabular}\end{flushleft}
and over $\F_{4294967311}$:
\begin{flushleft}\begin{tabular}{lll}
   $Y^2$&$=$&$2420332800X^5 + 3653091983X^4 + 2536585478X^3 + 2805510580X^2 +$\\ &&$ 159741347X + 2690010753$,\\
   $Y^2$&$=$&$4076826784X^5 + 2616936853X^4 + 3748957676X^3 + 1209100179X^2 +$\\ &&$ 3172892980X + 1266950302$.
  \end{tabular}\end{flushleft}
Finally, we give two pairs of curves with $(7,7)$-isogenous Jacobians, computed  using the modular polynomials with the $b'_i$. Over $\F_{10009}$:
\begin{flushleft}\begin{tabular}{lll}
   $Y^2$&$=$&$4826X^5 + 471X^4 + 2876X^3 + 5411X^2 + 7948X + 1308$,\\
   $Y^2$&$=$&$7218X^5 + 7699X^4 + 7011X^3 + 7103X^2 + 1845X + 4087$
  \end{tabular}\end{flushleft}
and over $\F_{3452678353}$:
\begin{flushleft}\begin{tabular}{lll}
   $Y^2$&$=$&$393356368X^5 + 1698662093X^4 + 471351782X^3 + 448279016X^2 +$\\ &&$1342046779X + 3241061457$,\\
   $Y^2$&$=$&$2171506943X^5 + 2231412358X^4 + 2005208933X^3 + 580698082X^2 +$\\ &&$306153493X + 474327543$.\\
  \end{tabular}\end{flushleft}

The motivated reader can check the curves we have constructed with the modular polynomials: it is enough to verify that the curves have the same zeta functions.

\paragraph{}
The polynomials are accessible at the adress: \phantom{a}  \url{http://www.math.u-bordeaux1.fr/~emilio/}.

%dire quelque part, on note parfois Phi_1p(X)=Phi_1p(X,f1,f2,f3)

\section*{Acknowledgements}
 Experiments presented in this paper were carried out using the PLAFRIM experimental testbed, being developed under the Inria PlaFRIM development action with support from LABRI and IMB and other entities: Conseil Régional d'Aquitaine, FeDER, Université de Bordeaux and CNRS (see https://plafrim.bordeaux.inria.fr/).
The author thank his PhD supervisors Andreas Enge and Damien Robert not only for the fruitful discussions we had, but also for their support and encouragement during his studies.
The author also thank the anonymous reviewer for his careful reading and his comments and suggestions.
This research was partially funded by ERC Starting Grant ANTICS 278537.

\bibliography{biblio}
\nocite{*}

%aff(P)={my(pol,test);test=0;pol="";for(ii=0,poldegree(P,x),for(jj=0,poldegree(P,y),for(kk=0,poldegree(P,z),i=poldegree(P,x)-ii;j=poldegree(P,y)-jj;k=poldegree(P,z)-kk;a=polcoeff(polcoeff(polcoeff(P,i,x),j,y),k,z);if(a!=0,if(a<0,test=1;if(a!=-1,pol=concat(pol,a),pol=concat(pol,"-")),if(test==1,pol=concat(pol,"+"));if(a!=1,test=1;pol=concat(pol,a)));if(i!=0,pol=concat(pol,"x");if(i!=1,pol=concat(pol,Str("^{",i,"}"))));if(j!=0,pol=concat(pol,"y");if(j!=1,pol=concat(pol,Str("^{",j,"}"))));if(k!=0,pol=concat(pol,"z");if(k!=1,pol=concat(pol,Str("^{",k,"}"))))))));return(pol);}

\end{document}